\documentclass[reqno,a4paper,12pt]{amsart} 
    
    \usepackage{amsmath,amscd,amsfonts,amssymb}
    \usepackage{mathrsfs,dsfont}
    \usepackage{color}
    \usepackage{mathtools}
    \usepackage{hyperref}
    \usepackage{tikz-cd}
    \usepackage[normalem]{ulem}

    \usepackage{tikz}
    \usetikzlibrary{decorations.pathreplacing}
    
    \numberwithin{equation}{section}
    \numberwithin{figure}{section}
    
    \def\R{\mathbb{R}}
    \def\C{\mathbb{C}}
    \def\Q{\mathbb{Q}}
    \def\Z{\mathbb{Z}}
    \def\N{\mathbb{N}}

    \def\one{\mathds{1}}

    \renewcommand\leq{\leqslant}
    \renewcommand\geq{\geqslant}
    
    \renewcommand\hat{\widehat}

    \newcommand{\supp}{\operatorname{supp}}

    \newcommand{\Hom}{\operatorname{Hom}}

    \newcommand{\p}{{\operatorname{p}}}
    \renewcommand{\c}{{\operatorname{c}}}
    \newcommand{\cc}{{(\operatorname{c})}}
    \newcommand{\e}{{(\operatorname{e})}}

    \theoremstyle{plain}
    \newtheorem{thm}{Theorem}[section]
    \newtheorem{theorem}[thm]{Theorem}
    
    \newtheorem{lemma}[thm]{Lemma}
    \newtheorem{corollary}[thm]{Corollary}
    
    \newtheorem{proposition}[thm]{Proposition}
    
    \newtheorem{question}[thm]{Question}
    
    \newtheorem{conjecture}[thm]{Conjecture}
    
    \newtheorem*{claim*}{Claim}

    \theoremstyle{definition}
    \newtheorem{definition}[thm]{Definition}
    \newtheorem*{definition*}{Definition}
    \newtheorem*{remarks*}{Remarks}
    \newtheorem*{remark*}{Remark}
    \newtheorem{remark}[thm]{Remark}
    \newtheorem{example}[thm]{Example}
    \newtheorem{examples}[thm]{Examples}

\begin{document}

	\title{Some variants of the periodic tiling conjecture}

	\author{Rachel Greenfeld}
	\address{Department of Mathematics, Northwestern University, Evanston, IL 60208.}
	\email{rgreenfeld@northwestern.edu}
 
	\author{Terence Tao}
	\address{UCLA Department of Mathematics, Los Angeles, CA 90095-1555.}
	\email{tao@math.ucla.edu}

	\date{}
	
	
	\begin{abstract}  The periodic tiling conjecture (PTC) asserts, for a finitely generated Abelian group $G$ and a finite subset $F$ of $G$, that if there is a set $A$ that solves the \emph{tiling equation} $\one_F * \one_A = 1$, there is also a periodic solution $\one_{A_{\p}}$.  This conjecture is known to hold for some groups $G$ and fail for others.  In this paper we establish three variants of the PTC.  The first (due to Tim Austin) replaces the constant function $1$ on the right hand side of the tiling equation by $0$, and the indicator functions $\one_F$ and $\one_A$ by bounded integer-valued functions.  The second, which applies in $G=\Z^2$, replaces the right hand side of the tiling equation by an integer-valued periodic function, and the functions $\one_F$ and $\one_A$ on the left hand side by bounded integer-valued functions.  The third (which is the most difficult to establish) is similar to the second, but retains the property of both $\one_A$ and $\one_{A_{\p}}$ being indicator functions; in particular, we establish the PTC  for multi-tilings in $G=\Z^2$.  As a result, we obtain the decidability of constant-level integer tilings in any finitely generated Abelian group $G$ and multi-tilings in $G=\Z^2$.
 \end{abstract}
	
	\maketitle
	
\section{Introduction}

Let $G = (G,+)$ be a finitely generated Abelian group, which we endow with the discrete topology; thus $G$ is isomorphic to $\Z^d \times H$ for some rank (or dimension) $d \geq 0$ and a finite Abelian group $H$.  We let $\ell^\infty(G, \C)$ denote the usual Banach space of bounded functions $f \colon G \to \C$; in particular, the indicator function $\one_A$ of any subset $A$ of $G$ lies in $\ell^\infty(G,\C)$.  Inside this space, we isolate two subspaces: the subspace $\ell^\infty(G, \C)_{\c}$ of functions of finite (or equivalently, compact) support; and the subspace $\ell^\infty(G, \C)_{\p}$ of functions $g \in \ell^\infty(G,\C)$ that are periodic in the sense that they are $\Gamma$-periodic (i.e., $g(x+h)=g(x)$ for all $x \in G$ and $h \in \Gamma$) for some lattice\footnote{In this discrete context, a lattice is simply a finite index subgroup of $G$.} $\Gamma$ of $G$.  Given $f \in \ell^\infty(G, \C)_{\c}$ and $g \in \ell^\infty(G, \C)$, we define the convolution $f*g \in \ell^\infty(G, \C)$ by the usual formula
$$ f*g(x) \coloneqq \sum_{y \in G} f(y) g(x-y),$$
noting that only finitely many summands are non-zero.  It is clear that this makes $\ell^\infty(G,\C)_{\c}$ a commutative complex algebra (indeed it is isomorphic to the group algebra $\C[G]$), that $\ell^\infty(G,\C)$ is a 
$\ell^\infty(G,\C)_{\c}$-module, and $\ell^\infty(G,\C)_{\p}$ is a $\ell^\infty(G,\C)_{\c}$-submodule of $\ell^\infty(G,\C)$.  We also define the subgroup $\ell^\infty(G,\Z)$ of $\ell^\infty(G,\C)$ consisting of bounded integer-valued functions on $G$, and define  $\ell^\infty(G, \Z)_{\c}$ and 
$\ell^\infty(G, \Z)_{\p}$ similarly.  Thus $\ell^\infty(G, \Z)_{\c}$ is isomorphic as a commutative ring to $\Z[G]$, $\ell^\infty(G,\Z)$ is a $\ell^\infty(G, \Z)_{\c}$-module, and $\ell^\infty(G, \Z)_{\p}$ is a $\ell^\infty(G, \Z)_{\c}$-submodule of $\ell^\infty(G,\Z)$.  Similarly for the integers $\Z$ replaced by the reals $\R$ or the rationals $\Q$; see Figure \ref{fig:inclusions}.

\begin{figure}
    \centering
\[\begin{tikzcd}
	{\ell^\infty(G,{\mathbb Z})_{\c}} && {\ell^\infty(G,{\mathbb Q})_{\c}} && {\ell^\infty(G,{\mathbb R})_{\c}} && {\ell^\infty(G,{\mathbb C})_{\c}} \\
	\\
	{\ell^\infty(G,{\mathbb Z})} && {\ell^\infty(G,{\mathbb Q})} && {\ell^\infty(G,{\mathbb R})} && {\ell^\infty(G,{\mathbb C})} \\
	\\
	{\ell^\infty(G,{\mathbb Z})_{\p}} && {\ell^\infty(G,{\mathbb Q})_{\p}} && {\ell^\infty(G,{\mathbb R})_{\p}} && {\ell^\infty(G,{\mathbb C})_{\p}}
	\arrow[hook, from=1-1, to=1-3]
	\arrow[hook, from=1-1, to=3-1]
	\arrow[hook, from=1-3, to=1-5]
	\arrow[hook, from=1-3, to=3-3]
	\arrow[hook, from=1-5, to=1-7]
	\arrow[hook, from=1-5, to=3-5]
	\arrow[hook, from=1-7, to=3-7]
	\arrow[hook, from=3-1, to=3-3]
	\arrow[hook, from=3-3, to=3-5]
	\arrow[hook, from=3-5, to=3-7]
	\arrow[hook, from=5-1, to=3-1]
	\arrow[hook, from=5-1, to=5-3]
	\arrow[hook, from=5-3, to=3-3]
	\arrow[hook, from=5-3, to=5-5]
	\arrow[hook, from=5-5, to=3-5]
	\arrow[hook, from=5-5, to=5-7]
	\arrow[hook, from=5-7, to=3-7]
\end{tikzcd}\]
    \caption{Inclusions between the various Abelian groups of bounded functions on $G$ studied in this paper.}
    \label{fig:inclusions}
\end{figure}

We will be interested in solving the equation
$$ f * a = g$$
for a given $f \in \ell^\infty(G,\Z)_{\c}$ and $g \in \ell^\infty(G,\Z)_{\p}$, and some $a$ that will lie in $\ell^\infty(G,\Z)$ and will often have additional constraints imposed, such as being periodic or being an indicator function $\one_A$.  A motivating problem in this area is the \emph{periodic tiling conjecture} (PTC), which we phrase using the above notation as follows:

\begin{conjecture}[Periodic tiling conjecture]\label{ptc}  Let $G$ be a finitely generated Abelian group, let $\one_F \in \ell^\infty(G,\Z)_{\c}$ be an indicator function in $\ell^\infty(G,\Z)_{\c}$, and suppose that the equation $\one_F * \one_A = 1$ has a solution $\one_A$ that is an indicator function in $\ell^\infty(G,\Z)$.  Then there also exists a solution $\one_{A_{\p}}$ to the equation $\one_F * \one_{A_{\p}} = 1$ that is an indicator function in $\ell^\infty(G,\Z)_{\p}$ (i.e., $A_{\p}$ is a periodic subset of $G$).
\end{conjecture}

Informally, the PTC suggests that if a finite set $F$ tiles $G$ by translations, then it also tiles $G$ periodically by translations.  

 The following partial results supporting Conjecture \ref{ptc} (or its Euclidean space version) are known:

\begin{itemize}
\item Conjecture \ref{ptc} is trivial when $G$ is a finite Abelian group, since in this case all subsets of $G$ are periodic.
\item  Conjecture  \ref{ptc}  was established for  $G=\Z$ and $G=\R$ \cite{N,LM,LW}.  
\item Conjecture \ref{ptc} holds for $G = \Z \times G_0$, where $G_0$ is any  finite Abelian group  \cite[Section 2]{GT23}. 
\item For $G=\Z^2$, Conjecture \ref{ptc}  was established by Bhattacharya \cite{BH} using ergodic theory methods.  In  \cite{GT21} we gave an alternative proof of this result, and  furthermore showed that every  tiling in $\Z^2$ by a single tile is \textit{weakly periodic}\footnote{A finite disjoint union of singly-periodic sets.}. 
\item When $G=\R^2$, Conjecture \ref{ptc} is known to hold for any tile that is a topological disk  \cite{bn,gbn,ken,err}.
\item In \cite{dggm}, Conjecture \ref{ptc} was also established for all rational polygonal sets in $\R^2$. 
\item Conjecture \ref{ptc} is known to be true for convex tiles in all dimensions  \cite{V,M}.
\item Conjecture \ref{ptc} is known to hold in $\Z^d$ when the size $|F|$ of $F$ is prime or equal to $4$ \cite{szegedy}. 
\item In \cite{bgu}, it was shown that the  periodic tiling conjecture in $\Z^d$  implies the  periodic tiling conjecture in every quotient group $\Z^d/\Lambda$.
\item In \cite{GT22}, we showed that the  periodic tiling conjecture in $\R^d$  implies the  periodic tiling conjecture in $\Z^d$.
\end{itemize}

Despite these positive evidences towards Conjecture \ref{ptc}, we recently established that the conjecture is false in sufficiently high dimensions \cite{GT22}. Moreover, in \cite{gk}, Kolountzakis and the first author showed that for sufficiently large $d$, there is a \emph{connected} (and open) counterexample to the periodic tiling conjecture in $\R^d$.

In this paper we establish some variants of the PTC.  We begin with an analysis of the equation\footnote{We remark that the equation $f*a=0$ plays a crucial role in the study of Nivat's conjecture \cite{Nivat} on the structure of low complexity configurations; see, for instance \cite{KS, KS20}.} $f*a=0$, where $a$ is not required to be an indicator function.  Since we always have the degenerate solution $f*0=0$ to this equation, we will restrict attention to $a$ that are not identically zero. To motivate the result, let us first recall a simple application of the Fourier transform:

\begin{theorem}[Solving $f*a=0$ in the complex numbers]\label{main-complex}  Let $G$ be a finitely generated Abelian group, and let $f \in \ell^\infty(G,\C)_{\c}$.  Then the following are equivalent:
\begin{itemize}
    \item[(i)]  One has $f*a=0$ for some $a \in \ell^\infty(G,\C) \backslash \{0\}$.
    \item[(ii)]  There exists a character $\chi \in \hat G$ (that is, a homomorphism $\chi \colon G \to S^1$ to the complex unit circle) such that the Fourier coefficient $\hat f(\chi) \coloneqq \sum_{x \in G} f(x) \overline{\chi(x)}$ vanishes.
\end{itemize}
\end{theorem}

\begin{proof} If (ii) holds, then $f * \chi = 0$, so (i) holds as well.  Conversely, if (ii) fails, then $\hat f$ is a nowhere vanishing trigonometric polynomial on $\hat G$ (which is the product of a torus and a finite Abelian group), hence it has a smooth inverse $1/{\hat f}$, which is then the Fourier transform of some absolutely integrable\footnote{This also follows from Wiener's $1/f$ theorem, although here the situation is much simpler as the Fourier transform is known to be smooth.} function $F$ by the Fourier inversion formula.  If $f*a=0$ for some $a \in \ell^\infty(G,\C)$, then by convolving both sides with $F$ and using the distributional Fourier transform we conclude that $a=0$, so (i) fails as well.
\end{proof}

In Section \ref{integer-sec} we provide an argument, due to Tim Austin \cite{T23}, establishing an analogous statement for the integers:

\begin{theorem}[Solving $f*a=0$ in the integers]\label{main-1}  Let $G$ be a finitely generated Abelian group, and let $f \in \ell^\infty(G,\Z)_{\c}$.  Then the following are equivalent:
\begin{itemize}
    \item[(i)]  One has $f*a=0$ for some $a \in \ell^\infty(G,\Z) \backslash \{0\}$.
    \item[(ii)]  One has $f*a_{\p}=0$ for some $a_{\p} \in \ell^\infty(G,\Z)_{\p} \backslash \{0\}$.
    \item[(iii)]  There exists a character $\chi \in \hat G$ that is of finite order (thus $\chi^m=1$ for some positive integer $m$) and such that the Fourier coefficient $\hat f(\chi)$ vanishes.
\end{itemize}
\end{theorem}

Informally, the equivalence of (i) and (ii) asserts that the analogue of PTC for $f*a=0$ holds, if we drop the requirement that $a$ must be an indicator function.

In practice, Theorem \ref{main-1} gives an easy criterion to test whether the equation $f*a=0$ has non-trivial solutions.   We illustrate this with a simple example:

\begin{example}  Let $G = \Z$ and $f \coloneqq 3 \times \one_{\{-1,+1\}} - 2 \times \one_{\{0\}}$.  Then the characters $\chi$ take the form $x \mapsto e^{2\pi i \theta x}$ for some $\theta \in \R/\Z$, and are of finite order precisely when $\theta$ is rational.  Observe that $\hat f(\chi) = 6 \cos(2\pi \theta) - 2$.  Since $\cos^{-1}(\frac{1}{3})$ is irrational ($\frac{1}{3} + \frac{\sqrt{8}}{3}i$ is not an algebraic integer and thus not a root of unity), we conclude that assertion (iii) in Theorem \ref{main-1} does not hold.  In particular, by Theorem \ref{main-1}, there is no solution $a \in \ell^\infty(G,\Z)$ to the equation $f*a=0$ other than the zero function, despite the fact that the Fourier transform $f$ still vanishes at some (irrational) frequencies (which, by Theorem \ref{main-complex}, ensures a non-trivial solution to $f*a=0$ in $\ell^\infty(G,\C)$).
\end{example}

There is a well-known connection between the PTC and the decidability of tilings. Indeed, H. Wang \cite{wang} showed that if the PTC were true, then there is an algorithm that computes for any given finite $F\subset G$ whether $F$ is a tile of $G$ or not. Thus, e.g., from \cite{N}, \cite{GT23} and \cite{BH}, we obtain the decidability of tilings in $\Z$, $\Z\times G_0$ for any finite Abelian group $G_0$, and $\Z^2$. However, in \cite{GT25} we established the undecidability of translational monotilings (when the finitely generated Abelian group $G$ is not fixed, but is part of the input of the algorithm). 

In Section \ref{sec:decidability}, by combining Theorem \ref{main-1} with known results on the structure of vanishing sums of roots of unity, together with Szmielew's theorem \cite{Szmielew} on the decidability of the first-order language of divisible Abelian groups, we prove that the solvability of equations $f*a=0$ in the integers, for $f \in \ell^\infty(G,\Z)_{\c}$ with $G$ being a finitely generated Abelian group, is decidable:

 \begin{corollary}[Solving $f*a=0$ in the integers is decidable]\label{cor:decidability-1}
     There exists an algorithm that computes in finite time, when given
     \begin{itemize}
         \item a finitely generated\footnote{For the purposes of computability, this group should be presented in the standard form $G = \Z^d \times \prod_{n=d+1}^r \Z/N_n\Z$ of a product of a finite number of copies of $\Z$ and cyclic groups.} Abelian group $G$; and
         \item a function $f \in \ell^\infty(G,\Z)_{\c}$,
     \end{itemize}
     whether there exists a non-zero $a\in \ell^\infty(G,\Z)$ such that $f*a=0$.
 \end{corollary}

 \begin{remark} If $f \in \ell^\infty(G,\Z)_{\c}$ has a non-zero sum, i.e., $f*1 \neq 0$, then the solvability of $f*a=0$ for some non-zero $a \in \ell^\infty(G,\Z)$ is clearly equivalent to the solvability of $f*a=k$ for some integer $k$ and some non-constant $a \in \ell^\infty(G,\Z)$, since $f*a=k$ if and only if $f*((f*1)a - k) = 0$.  Similarly, the solvability of $f*a_{\p}=0$ for some non-zero $a_{\p} \in \ell^\infty(G,\Z)_{\p}$ is equivalent to the solvability of $f*a_{\p}=k$ for some integer $k$ and some non-constant $a_{\p} \in \ell^\infty(G,\Z)_{\p}$.  Thus, when $f$ has a non-zero sum, one can add some further equivalences to the three listed in Theorem \ref{main-1}.  Conversely, if $f$ has a zero sum, then by the amenability of $G$ there are no solutions to $f*a=k$ for any non-zero integer $k$ and any $a \in \ell^\infty(G,\C)$.
\end{remark}

In the case of rank one groups $G = \Z \times H$ with $H$ finite, periodicity of solutions to $f*a=g$ is easy to ensure:

\begin{proposition}[Periodicity of integer tilings in rank one groups]\label{ptc-1d}  Let $G = \Z \times H$ for some finite $H$, $f \in \ell^\infty(G,\Z)_{\c} \backslash \{0\}$, and $g \in \ell^\infty(G,\Z)_{\p}$.  If the equation $f*a=g$ has a solution $a$ in $\ell^\infty(G,\Z)$, then it also has a solution $a_{\p}$ in $\ell^\infty(G,\Z)_{\p}$.
\end{proposition}

We prove this proposition (a routine generalization of \cite[Theorem 2.1]{GT23}, though with a different proof) in Section \ref{structure-sec}.  Of course, the corresponding statement for rank zero (i.e., finite) groups is trivial, since $\ell^\infty(G,\Z) = \ell^\infty(G,\Z)_{\p}$ in this case.

Our remaining results are specific to the group $\Z^2$.  In Section \ref{integer-ptc}, we establish a version of the PTC when the requirement of being an indicator function is replaced with the requirement of being in $\ell^\infty(\Z^2,\Z)$:

\begin{theorem}[Periodicity of integer tilings in $\Z^2$]\label{main-2}  Let $f \in \ell^\infty(\Z^2,\Z)_{\c}$, and let $g \in \ell^\infty(\Z^2,\Z)_{\p}$.  If the equation $f*a=g$ has a solution $a$ in $\ell^\infty(\Z^2,\Z)$, then it also has a solution $a_{\p}$ in $\ell^\infty(\Z^2,\Z)_{\p}$.
\end{theorem} 

Finally, we establish a PTC for indicator functions in $\Z^2$, in which the periodic level function $g$ is not required to be an indicator function:

\begin{theorem}[Periodicity of multi-tilings in $\Z^2$]\label{main-3}  Let $f \in \ell^\infty(\Z^2,\Z)_{\c}$, and let $g \in \ell^\infty(\Z^2,\Z)_{\p}$.  If the equation $f *a =g$ has an indicator function solution $a=\one_A$ in $\ell^\infty(\Z^2,\Z)$, then it also has an indicator function solution $a_\p=\one_{A_{\p}}$ in $\ell^\infty(\Z^2,\Z)_{\p}$.
\end{theorem}

In particular, the periodic tiling conjecture in $\Z^2$ for level $k$ tilings $\one_F * \one_A = k$ holds for any natural number $k$.  This is despite the fact that there exist tilings at higher level which are not weakly periodic (the finite union of sets periodic in a single direction); see \cite[Section 2]{GT21}.  As mentioned earlier, the special case $k=1$ of this theorem was established in \cite{BH}, \cite{GT21}.

A standard compactness argument (originally due to H. Wang \cite{wang}) gives,  as a corollary of Theorem \ref{main-3}, the decidability of multi-tilings in $\Z^2$:

\begin{corollary}[Multi-tilings in $\Z^2$ are decidable]\label{cor:decidability-3}
    There exists an algorithm that computes in finite time, upon any given $f\in \ell^\infty(\Z^2,\Z)_{\c}$ and  $g \in \ell^\infty(\Z^2,\Z)_{\p}$, whether there exists a set $A\subset \Z^2$ such that $f *\one_A =g$.
\end{corollary}

We pose the analogous version of Theorem \ref{main-2} as an open problem:

\begin{question} Is there an algorithm that computes in finite time, upon any given $f\in \ell^\infty(\Z^2,\Z)_{\c}$ and  $g \in \ell^\infty(\Z^2,\Z)_{\p}$, whether there exists $a \in \ell^\infty(\Z^2,\Z)$ such that $f * a =g$?
\end{question}

Theorem \ref{main-2} allows one to assume that $a$ is periodic in the above question.  However, the compactness argument from \cite{wang} does not directly apply because we have no computable uniform bound on $a$.

The results in \cite{GT22} imply that the analogue of Theorem \ref{main-3} will fail in sufficiently high dimension.  However, the analogous question for Theorem \ref{main-2} remains open:

\begin{question}   Let $f \in \ell^\infty(\Z^d,\Z)_{\c}$ and $g \in \ell^\infty(\Z^d,\Z)_{\p}$.  If the equation $f*a=g$ has a solution $a$ in $\ell^\infty(\Z^d,\Z)$, does it also has a solution $a_{\p}$ in $\ell^\infty(\Z^2,\Z)_{\p}$?
\end{question}

\subsection{Methods of proof}

All of our main results rely on a \emph{dilation lemma} (see \cite[Proposition 3.1]{BH}, \cite[Lemma 3.1]{GT21}, \cite[Theorem 1.2(i)]{ggrt}), which establishes the remarkable fact that the identity $f*a=g$ is stable under certain dilations of the function $f$.  By averaging over dilations, this leads to the \emph{structure theorem} (see \cite[Theorem 3.3]{BH}, \cite[Theorem 1.7]{GT21}, \cite[Theorem 1.2(ii)]{ggrt}) that asserts that solutions $a$ to $f*a=g$ have a special form.  For instance, in the group $\Z^2$, one has a decomposition
$$ a = \tilde g - \sum_{w \in W} \varphi_w$$
where $\tilde g$ is periodic, and each $\varphi_w$ is $qw$-periodic for a different primitive direction $w$ and some ``sufficiently divisible'' $q$ (see Theorem \ref{structure-thm-2} for a precise statement).  An additional useful property is the ``slicing lemma''\footnote{See also \cite[Lemma 5.1]{GT21}.} on $\Z^2$, which asserts that the convolutions $(\one_{x+\langle w\rangle} f)*\varphi_w$ of $\varphi_w$ with various ``slices'' $\one_{x+\langle w\rangle} f$ of $f$ are periodic.   We lay out our formulation of the dilation lemma, structure theorem, and slicing lemma in Section \ref{structure-sec}.

This structural theory is easiest to apply for the homogeneous equation $f*a=0$, because the solution space to this equation is invariant under difference operators $\partial_h$, which can be used to largely eliminate the role of the partially periodic functions $\varphi_w$.  As such, the proof of Theorem \ref{main-1} is quite short, and is given in Section \ref{integer-sec}.

Then, in Section \ref{sec:decidability}, based on Theorem \ref{main-1}, we prove Corollary \ref{cor:decidability-1} by using the theory of vanishing sums of roots of unity to show that the character $\chi$ in Theorem \ref{main-1}(iii) can be chosen involves roots of unity whose order can be controlled by a computable quantity $M=M(f)$ depending only on $f$. 

The case of integer-valued solutions $a \in \ell^\infty(\Z^2,\Z)$ to $f*a=g$ on the group $\Z^2$ is also relatively easy.  Here, one can use repeated differentiation to show that the functions $\varphi_w$ mentioned previously behaves in a polynomial fashion modulo $1$.  The coefficients of these polynomials can be irrational, leading to non-periodic behavior; however, by applying a retraction homomorphism from $\R$ to $\Q$, one can map these coefficients to rational numbers, thus effectively making $\varphi_w$ periodic modulo one.  Some further modification of the $\varphi_w$, utilizing the slicing lemma (Theorem \ref{structure-thm-2}(ii)), can then be performed to make $\varphi_w$ genuinely periodic, as opposed to merely being periodic modulo one; see Section \ref{integer-ptc}.

The case of indicator function solutions $f*\one_A=g$ on $\Z^2$ is significantly harder, and requires a much finer analysis of the structure of these solutions in order to construct a solution $a_\p = \one_{A_\p}$ that is simultaneously periodic and an indicator function.  As it turns out, there are two primary sources of non-periodicity in these indicator function solutions.  The first source might be termed ``combinatorial'' non-periodicity, arising from sets $A$ that contain subsets $A_{x_0}$ that are periodic in just one direction.  A typical example occurs, e.g.,  when $f = \one_{\{(0,0), (1,0)\}}$, $g = 1$, and $A$ is a set of the form
$$ A \coloneqq \{ (2n+a(m),m) \in \Z^2: n,m\in \Z \},$$
for some arbitrary function $a \colon \Z \to \{0,1\}$. Then we have $f * \one_A = 1$, but $A$ only exhibits periodicity in the direction $(2,0)$ in general.  A second source of non-periodicity, first identified in \cite[Section 2]{GT21}, comes from equidistributed functions such as
$$ (n,m) \mapsto \{ \alpha (an+bm) + \beta \}$$
for some irrational numbers $\alpha,\beta$ and integer coefficients $a,b$.  In the higher level case when $g$ can exceed $1$, it is possible for the $\varphi_w$ to be of this form on certain cosets of $\Z^2$; we refer the reader to \cite[Section 2]{GT21} for an explicit example of the form $\one_F * \one_A = 4$.  

Using the second moment method (as well as a little bit of probability theory in order to avoid some rare bad events), it is possible to show that these two scenarios are in some sense the \emph{only} sources of non-periodicity: a solution to the equation $f * \one_A = g$ can be described in terms of some combinatorial data $A_{x_0}$ (that is periodic in one direction), as well as some ``equidistributed'' data such as the real coefficients $\alpha$, $\beta$.  This analysis extends the analysis in \cite{GT21}, which handled the case where $g=1$.  In that special case, the equidistributed case could be ruled out by \emph{ad hoc} methods; but here, the equidistributed case can very much exist, and requires new arguments to treat.

The equidistributed data $\alpha,\beta$ may be required to obey some linear inequalities in order to generate a solution $f * \one_A = g$.  Unfortunately, such inequalities are not necessarily preserved by retraction homomorphisms, so the technique used to prove Theorem \ref{main-2} does not seem to be available in this case.  Nevertheless, it is possible to eliminate the role of the combinatorial data, and perform some ``quantifier elimination'' on the constraints on the equidistributed data, to reduce those constraints to a (locally) finite boolean combination of linear inequalities on this data (with rational coefficients).  One can then appeal to a general ``rationalization'' lemma (of a model-theoretic flavor) to replace this data with rational numbers, which ends up leading to a periodic solution, thus establishing Theorem \ref{main-3}.  The details of this argument are provided in Sections \ref{cleaning-sec}, \ref{highlevel-sec}, after some preliminaries\footnote{There is one technical issue that arises when executing this strategy, namely that one has to avoid the discontinuities of the fractional part operator $\{\}$.  This turns out to be possible using some elementary $p$-adic techniques, although it does cause some technical complications to the arguments; see Proposition \ref{nondeg-real}.} in Section \ref{rationalization-sec}.

\subsection{Acknowledgments}

RG was partially supported by the Association of Members of the Institute for Advanced Study (AMIAS) and by NSF grant DMS-2242871 and DMS-2448416. TT was supported by NSF grant DMS-1764034 and DMS-2347850.

We are grateful to Tim Austin for initiating the discussion about integer tilings and for sharing his proof of Theorem \ref{main-1} (provided in Section \ref{integer-sec}).

\section{Notation}

Given $h \in G$, we define the difference operator $\partial_h \colon \ell^\infty(G,\C) \to \ell^\infty(G,\C)$ by the formula
$$ \partial_h f(x) \coloneqq f(x+h)-f(x).$$
Equivalently, $\partial_h$ is convolution with $\one_{\{h\}} - \one_{\{0\}}$.  In particular, the homomorphisms $\partial_h$ preserves all the spaces in Figure \ref{fig:inclusions}, and commute with each other and with convolutions.

If $h \in G$, we write $\langle h \rangle \coloneqq \{ nh: n \in \Z \}$ for the cyclic group generated by $h$, and abbreviate ``$\langle h \rangle$-periodic'' as ``$h$-periodic''.  Thus, a function $f$ is $h$-periodic if and only if it is annihilated by $\partial_h$.

We define a wedge product $\wedge \colon \Z^2 \times \Z^2 \to \Z$ by the formula
$$ (a,b) \wedge (c,d) \coloneqq ad-bc.$$
An element $w$ of $\Z^2$ will be called \emph{primitive} if it is non-zero and is not of the form $mw'$ for some $m \geq 2$ and $w' \in \Z^2$. Note that every non-zero element of $\Z^2$ is an integer multiple of a primitive element.  If $w \in \Z^2$ is primitive, then by Bezout's theorem we can find a complementary primitive vector $w^* \in \Z^2$ such that $w \wedge w^* = 1$.  This implies that $w$ and $w^*$ form a basis of generators for $\Z^2$ as an Abelian group; indeed for any $y \in \Z^2$ we see from Cram\'er's rule that
\begin{equation}\label{ynm}
 y = (w \wedge y) w^* - (w^* \wedge y) w 
 \end{equation}
The choice of complementary vector $w^*$ is not quite unique (it is determined up to a shift in $\langle w \rangle$), but we will arbitrarily choose some recipe for producing $w^*$ from $w$ and fix it henceforth.  

We will frequently work with various moduli $q \geq 1$, which we will order by divisibility.  In particular, we will use ``for $q$ sufficiently divisible'' as shorthand for ``for $q$ divisible by an appropriate factor'', and ``making $q$ more divisible'' as shorthand for ``replacing $q$ by a suitable multiple of $q$''.  Thus, for instance, a function $f \in \ell^\infty(G,\C)$ is periodic if and only if it is $qG$-periodic for sufficiently divisible $q$. 

A basic tool for building periodic functions for us will be that of periodically extending off of a fundamental domain.

\begin{definition}\label{period-def}  Let $\Gamma$ be a subgroup of an Abelian group $G$.  A \emph{fundamental domain} $\Omega$ for $G/\Gamma$ is a subset of $G$ that meets every coset $x+\Gamma$ of $G$ in exactly one point.  Define the projection operator $\Pi_{\Omega,\Gamma} \colon G \to G$ by setting $\Pi_{\Omega,\Gamma}(g)$ to be the unique element $\omega$ of $\Omega$ such that $g \in \omega + \Gamma$; equivalently, $\Pi_{\Omega,\Gamma}$ is the unique $\Gamma$-periodic map that is the identity on $\Omega$.  More generally, for any function $f: G \to X$, $f \circ \Pi_{\Omega,\Gamma} \colon G \to X$ is the unique $\Gamma$-periodic function that agrees with $f$ on $\Omega$.
\end{definition}

As a simple application of this construction, we observe the following claim, that shows that the analogue of the PTC for function composition is trivially true:

\begin{lemma}[PTC for composition]\label{period-extend}  Let $\Gamma$ be a subgroup of an Abelian group $G$, let $g \colon G \to Y$ be a $\Gamma$-periodic function, and let $F \colon X \to Y$ be another function.  If the equation $F \circ a = g$ has a solution $a \colon G \to X$, then it also has a $\Gamma$-periodic solution $a_\p \colon G \to X$.
\end{lemma}

\begin{proof}  Take $a_\p \coloneqq a \circ \Pi_{\Omega,\Gamma}$, where $\Omega$ is an arbitrarily chosen fundamental domain for $G/\Gamma$.  Then $a_\p$ is $\Gamma$-periodic, and
$$ F \circ a_\p = F \circ a \circ \Pi_{\Omega,\Gamma} = g \circ \Pi_{\Omega,\Gamma} = g.$$
\end{proof}

\begin{remark} This lemma can be used to simplify the proof of \cite[Corollary 5.3]{GT21}, as the periodic function $z \mapsto \tilde A_z$ in that proof can be constructed directly from the function $z \mapsto A_z$ by means of this lemma.
\end{remark}

We write $[q]$ for the set $\{0,\dots,q-1\}$; this is a  fundamental domain for $\Z/q\Z$.

We will often make use of \emph{retraction homomorphisms} $\psi \colon \R \to \Q$, which are $\Q$-linear maps (viewing $\R$ as a vector space over $\Q$) that are the identity on $\Q$.  Such homomorphisms exist thanks to Zorn's lemma (or the Hahn--Banach theorem), although they cannot be continuous or measurable, and do not map bounded sets to bounded sets.  One can similarly construct retraction homomorphisms $\psi \colon \C \to \Q$ on the complex numbers.

The fractional part $\{x\}$ of a real number $x$ is defined as $\{x\} \coloneqq x - \lfloor x \rfloor$, where $\lfloor x\rfloor$ is the greatest integer less than or equal to $x$; it is also the projection operator $\Pi_{[0,1),\Z}$.  The discontinuities of this function $\{\}$ at the integers will cause some technical difficulty, and we will invest some effort in avoiding them in the proof of Theorem \ref{main-3}.

Given a finite set $F$, we use $|F|$ to denote its cardinality.

We use the Hahn--Banach theorem\footnote{One could also use an ultrafilter here if desired.} to select a generalized limit functional $\widetilde{\lim} \colon \ell^\infty(\N, \R) \to \R$ that extends the usual limit functional on convergent sequences, is a linear functional of norm $1$, and is such that
\begin{equation}\label{contract}
\liminf_{n \to\infty} a_n \leq \widetilde{\lim}_{n \to \infty} a_n \leq \limsup_{n \to \infty} a_n
\end{equation}
for any bounded real sequence $a_n$.  Given any shift $v \in G$, we define the (one-sided) projection operator $\pi_v \colon \ell^\infty(G,\R) \to \ell^\infty(G,\R)$ by the formula
\begin{equation}\label{piv-def}
\pi_v a(x) \coloneqq \widetilde{\lim}_{N \to \infty} \frac{1}{N} \sum_{n=1}^N a(x+nv).
\end{equation}
From \eqref{contract} we see that $\pi_v$ is a contraction on $\ell^\infty(G,\R)$, that is the identity on $v$-periodic functions, and whose image is always $v$-invariant, thus $\partial_v \pi_v = 0$ and $\pi_v \pi_v = \pi_v$.  Also, $\pi_v$ commutes with convolutions by functions in $\ell^\infty(G,\R)_{\c}$, and thus by difference operators $\partial_w$ for any $w \in G$.  In particular, if $f \in \ell^\infty(G,\R)$ is $w$-periodic for some $w \in G$, then $\pi_v f$ is also $w$-periodic.  For $v=0$, the operator $\pi_v$ is simply the identity, but we will usually work with the case when the shift $v$ is non-zero.

\begin{remark}  With our current formalism, we cannot quite guarantee that the projection operators $\pi_v$ commute with each other, or that $\pi_v = \pi_{-v}$.  Such properties can be ensured (up to almost everywhere equivalence) by transferring to an ergodic theory framework and applying the ergodic theorem, but we will not need to do so here.  
\end{remark}

\section{The dilation lemma, and some consequences}\label{structure-sec}

Let $G$ be a finitely generated Abelian group.
Given a function $f \in \ell^\infty(G,\Z)_{\c}$, which we can write as $f = \sum_x f(x) \one_{\{x\}}$ since there are only finitely many non-zero summands, we can define the dilates $\tau_r f \in \ell^\infty(G,\Z)_{\c}$ for any natural number $r$ by the formula
$$ \tau_r f \coloneqq \sum_x f(x) \one_{\{rx\}}.$$
The following ``dilation lemma'' is a routine modification of existing dilation lemmas in the literature (see, e.g., \cite{tijdeman}, \cite[Proposition 3.1]{BH}, \cite[Lemma 3.1]{GT21}, \cite[Theorem 1.1]{ggrt}, \cite[Lemma 2]{KS}, \cite[Theorem 10]{szegedy}):

\begin{lemma}[Dilation lemma]\label{dil}  Let $f \in \ell^\infty(G,\Z)_{\c}$ and $g \in \ell^\infty(G,\Z)_{\p}$.  Then for sufficiently divisible $q$, if $f*a = g$ for some $a \in \ell^\infty(G,\Z)$, then $(\tau_r f)*a = g$ for all $r \geq 1$ that is equal to $1$ mod $q$.
\end{lemma}

\begin{proof}  We first consider the case $g=0$.  By the fundamental theorem of arithmetic, it suffices to establish the claim with $r$ replaced by all primes $p$ larger than some threshold $w$ (and then taking $q$ to be divisible by all primes up to $w$).  If $f*a=0$, then clearly $f^{*p}*a=0$, where $f^{*p}$ is the convolution of $p$ copies of $f$.  Since the Frobenius map $x \mapsto x^p$ is a homomorphism in characteristic $p$, we have
$$ f^{*p} = \tau_p f \mod p$$
and thus
$$ (\tau_p f) * a = 0 \mod p.$$
Note that $(\tau_p f) * a$ is bounded independently of $p$.  Thus, for $p$ large enough, we obtain $(\tau_p f) * a = 0$ as required.  We remark that $q$ can be chosen to depend on $a$ only through the $\ell^\infty$ norm of $a$.

Now we consider the case of general $g$. Since $g$ is assumed to be periodic, we may assume (for sufficiently divisible $q$) that $\partial_{qh} g = 0$ for all $h \in G$.  In particular, if $f*a=g$, then
$$ f * \partial_{qh} a = 0.$$
Since the $\ell^\infty$ norm of $\partial_{qh} a$ is bounded uniformly in $q$ and $h$, we conclude from the previous discussion that (for sufficiently divisible $q$) one has
$$ \tau_r f * \partial_{qh} a = 0$$
for any $r \geq 1$ coprime to $q$.  In particular, 
$$ \partial_{qh} ( (\tau_r f - f) * a ) =0,$$
that is to say $(\tau_r f - f) * a$ is periodic with respect to the finite index subgroup $qG$.  On the other hand, for $r = 1$ mod $q$, $\tau_r f - f$ has mean zero on every coset of $qG$.  As $qG$ is amenable, this implies that $(\tau_r f - f) * a$ also has mean zero on every such coset, and thus vanishes identically by periodicity.  We conclude that $\tau_r f * a = f * a = g$, as required.
\end{proof}

This leads to the following structure theorem.

\begin{theorem}[Structure theorem]\label{structure-thm} Let $G = \Z^d \times H$ for some $d\geq 1$ and finite Abelian $H$, and let $f \in \ell^\infty(G,\Z)_{\c}$.  For each $v \in \Z^d$, let $f_v \coloneqq \one_{\{v\} \times H} f$ be the restriction of $f$ to $\{v\} \times H$, and let $V \subset \Z^d$ be the (finite) set of $v$ for which $f_v$ is not identically zero, thus $f = \sum_{v \in V} f_v$.  Suppose that $f * a = g$ for some $a \in \ell^\infty(G,\Z)$ and $g \in \ell^\infty(G,\Z)_{\p}$.  Then for any $v_0 \in V$, one has the decomposition
$$ f_{v_0} * a = g - \sum_{w \in W_{v_0}} \varphi_{v_0,w}$$
where $W_{v_0}=W_{v_0}(f)$ is a finite set of linearly independent primitive elements of $\Z^d$ that depends only on $f$, and $\varphi_{v_0,w} \in \ell^\infty(G,\R)$ is $qw$-periodic for sufficiently divisible $q$ (where we view $\Z^d$ as a subgroup of $\Z^d \times H$).
\end{theorem}

\begin{proof} We adapt the proof of \cite[Theorem 1.7]{GT21}. By translation we may assume without loss of generality that $v_0=0$.  By Lemma \ref{dil}, we see for sufficiently divisible $q$ that
$$
(\tau_{1+qr} f) * a = g
$$
for all $r \geq 0$.  By making $q$ more divisible as necessary, we may assume that $q$ is a multiple of the order of $H$, so in particular $\tau_{1+qr} f_0 = f_0$, and more generally $\tau_{1+qr} f_v = \one_{\{qrv\}} * f_v$ for any $v$.  We then have
$$ f_0 * a = g - \sum_{v \in V \backslash \{0\}} \one_{\{qrv\}} * f_v * a.$$
Averaging in $r$ and taking generalized limits, we conclude that
$$ f_0 * a = g - \sum_{v \in V \backslash \{0\}} \pi_{qv} (f_v * a).$$
Each $v \in V \backslash \{0\}$ is a integer multiple of some primitive element $w$ in $\Z^d$. If we let $W_0=W_0(f)$ be the set of such $w$ obtained, and $\varphi_{0,w}$ the sum of all the $\pi_{qv} (f_v * a)$ associated to a given $w$, we obtain the desired decomposition
$$ f_{0} * a = g - \sum_{w \in W_{0}} \varphi_{0,w}$$
with each $\varphi_{0,w}$ $q'w$-periodic for some positive integer $q'$.
\end{proof}

As a first quick application of this theorem, we can now establish Proposition \ref{ptc-1d}.

\begin{proof}[Proof of Proposition \ref{ptc-1d}]  Let $H,f,g$ be as in the proposition, thus we have a solution $a \in \ell^\infty(\Z \times H, \Z)$ to $f*a=g$.  Decomposing $f = \sum_{v \in V} f_v$ as in Theorem \ref{structure-thm}, we conclude from the above theorem that for each $v_0 \in V$, $f_{v_0}*a$ is $q$-periodic for some sufficiently divisible $q$, which we can take to be independent of $v_0$ by taking least common multiples.  Let us now write this in a different way.  If we let $\tilde a \colon \Z \to \ell^\infty(H,\Z)$ be the function $\tilde a(j)(h) \coloneqq a(j,h)$, and similarly let $\tilde f_{v_0} \in \ell^\infty(H,\Z)$ be the function $\tilde f_{v_0}(h) = f_{v_0}(v_0,h) = f(v_0,h)$, then by the above discussion the map
$$ j \mapsto ( \tilde f_{v_0} * \tilde a(j) )_{v_0 \in V}$$
is a $q$-periodic map from $\Z$ to $\ell^\infty(H,\Z)^V$.  Applying Lemma \ref{period-extend}, we can thus find a $q$-periodic map $\tilde a_\p \colon \Z \to \ell^\infty(H,\Z)$ such that
$$ (\tilde f_{v_0} * \tilde a(j))_{v_0 \in V} = (\tilde f_{v_0} * \tilde a_\p(j))_{v_0 \in V}$$
for all $j \in \Z$.  If we then let $a_\p \in \ell^\infty(\Z \times H,\Z)_\p$ be defined by $a_\p(j,h) \coloneqq \tilde a_\p(j)(h)$, we see from construction that $a_\p$ is periodic and $f_{v_0}*a_\p = f_{v_0} * a$ for all $v_0 \in V$, hence on summing in $v_0$ we have $f * a_\p = f*a = g$ as required.
\end{proof}

Now we study the case $G=\Z^2$, where the structure theorem is also quite strong.  It will be convenient to make the following definition.

\begin{definition}[Structured solution]   Let $f \in \ell^\infty(\Z^2,\Z)_{\c} \backslash \{0\}$ and $g \in \ell^\infty(\Z^2,\Z)_{\p}$.  A \emph{structured solution} to the equation $f*a=g$ is a tuple $(a, \tilde g, W, (\varphi_w)_{w \in W})$, where $a \in \ell^\infty(\Z^2,\Z)$, $\tilde g \in \ell^\infty(\Z^2,\R)_{\p}$, $W$ is a finite set of linearly independent primitive elements of $\Z^2$, and for each $w \in W$, $\varphi_w \in \ell^\infty(\Z^2,\R)$ is $qw$-periodic for sufficiently divisible $q$, such that $a$ solves the equation
\begin{equation}\label{fag-eq}
f*a = g
\end{equation}
and has the representation
\begin{equation}\label{a-rep}
a = \tilde g - \sum_{w \in W} \varphi_w.
\end{equation}
If $a$ is an indicator function $\one_A$, we refer to $(\one_A, \tilde g, W, (\varphi_w)_{w \in W})$ as an \emph{indicator function structured solution}, thus in this case we have the equation
\begin{equation}\label{fag-eq-ind}
f*\one_A = g
\end{equation}
and the representation
\begin{equation}\label{a-rep-ind}
\one_A = \tilde g - \sum_{w \in W} \varphi_w.
\end{equation}
\end{definition}

\begin{theorem}[Structure theorem in $\Z^2$]\label{structure-thm-2} Let $f \in \ell^\infty(\Z^2,\Z)_{\c} \backslash \{0\}$, $a \in \ell^\infty(\Z^2,\Z)$, and $g \in \ell^\infty(\Z^2,\Z)_{\p}$ be such that $f*a=g$.  
\begin{itemize}
    \item[(i)]  (Structure theorem) There exists a structured solution $(a, \tilde g, W, (\varphi_w)_{w \in W})$ to $f*a=g$ with the given choice of $a$, obeying the additional property $\tilde g \in \ell^\infty(\Z^2,\Q)_{\p}$. 
    \item[(ii)]  (Slicing lemma) If $(a, \tilde g, W, (\varphi_w)_{w \in W})$ is a structured solution to $f*a=g$, then for any $w_0 \in W$ and any coset $x+\langle w_0 \rangle$ of the cyclic group $\langle w_0 \rangle$ generated by $w_0$, the function $(\one_{x + \langle w_0 \rangle} f) * \varphi_{w_0}$ is periodic.
\end{itemize}
\end{theorem}

\begin{proof} We begin with (i).  From Theorem \ref{structure-thm} applied to the case $H=\{0\}$ and some $v_0\in\Z^2$ with $f(v_0)$ non-zero, translating by $-v_0$, and then dividing by $f(v_0)$, we have
\begin{equation}\label{ag}
 a = \frac{1}{f(v_0)} g - \sum_{w \in W} \varphi_w
\end{equation}
where $W$ is a finite set of linearly independent primitive elements $w$ of $\Z^2$, and each $\varphi_w \in \ell^\infty(\Z^2,\R)$ is $qw$-periodic for sufficiently divisible $q$, which we can take to be independent of $w$ by taking least common multiples.  This gives (i).

Now we show (ii).  Applying a translation, we may assume without loss of generality that $x=0$.

We split $f = f_0 + \sum_{j=1}^J f_j$, where $f_0$ is the restriction of $f$ to $\langle w_0\rangle$, $J$ is finite, and each $f_j$ is the restriction of $f$ to some non-identity coset $x_j + \langle w_0 \rangle$.
Now we adapt the proof of \cite[Lemma 5.1]{GT21}.  Let $q$ be sufficiently divisible, then we see from Theorem \ref{dil} that $(\tau_{1+qr} f) * a = g$ for all $r \geq 1$.  Inserting \eqref{a-rep} and rearranging, we conclude that
$$ (\tau_{1+qr} f_0) * \varphi_{w_0}
= (\tau_{1+qr} f_0) * \tilde g - g - \sum_{j=1}^J (\tau_{1+qr} f_{j}) * \varphi_{w_0} - \sum_{j=0}^J \sum_{w \in W \backslash \{w_0\}} (\tau_{1+qr} f_{j}) * \varphi_w.$$
By the $qw_0$-periodicity of $\varphi_{w_0}$, we have
$$ (\tau_{1+qr} f_0) * \varphi_{w_0} = f_0 * \varphi_{w_0}$$
and
$$ (\tau_{1+qr} f_j) * \varphi_{w_0} = \one_{\{qrx_j\}} * f_j * \varphi_{w_0}$$
Applying $\pi_{qw_0}$, we then conclude
$$ f_0 * \varphi_{w_0}
=  (\tau_{1+qr} f_0) * \pi_{qw_0} \tilde g - \pi_{qw_0} g - \sum_{j=1}^J \one_{\{qrx_j\}} * f_j * \varphi_{w_0} - \sum_{j=0}^J\sum_{w \in W \backslash \{w_0\}} (\tau_{1+qr} f_{j}) * \pi_{qw_0} \varphi_w.$$
In two dimensions, $qw_0$ and $qw$ generate a finite index subgroup of $\Z^2$ for any $w \in W \backslash \{w_0\}$, so $\pi_{qw_0} \varphi_w$ is periodic.  Since $\tilde g$ and $g$ were also periodic, we conclude that
$$ f_0 * \varphi_{w_0}
= g^{(r)} - \sum_{j=1}^J \one_{\{qrx_j\}} * f_j * \varphi_{w_0}
$$
for some $g^{(r)} \in \ell^\infty(\Z^2,\R)_{\p}$, where the period of $g^{(r)}$ can be taken to be uniform in $r$.  Averaging over $r$ and taking generalized limits, we conclude that
$$ f_0 * \varphi_{w_0}
= g' - \sum_{j=1}^J  f_j * \pi_{qx_j} \varphi_{w_0}
$$
for some $g' \in \ell^\infty(\Z^2,\R)_{\p}$.  As $x_j + \langle w_0 \rangle$ is a non-identity coset, $qx_j$ and $qw_0$ generate a finite index subgroup of $\Z^2$, thus $\pi_{qx_j} \varphi_{w_0}$ also lies in $\ell^\infty(\Z^2,\R)_{\p}$ for each $j$.  Thus we have $f_0 * \varphi_{w_0} \in \ell^\infty(\Z^2,\R)_{\p}$ as claimed.
\end{proof}

\section{Solving \texorpdfstring{$f*a=0$}{f*a=0} in the integers}\label{integer-sec}

In this section we provide a proof, due to Tim Austin \cite{T23}, of Theorem \ref{main-1}.  It is clear that (ii) implies (i).  If (ii) holds, then $f*a_{\p}=0$ for some $a_{\p} \in \ell^\infty(G,\Z)_{\p} \backslash \{0\}$.  The distributional Fourier transform of the periodic function $a_{\p}$ is supported on a finite non-empty set of characters in $\hat G$, each of which is of finite order, thus $\hat f$ (which is a trigonometric polynomial, and hence smooth, on $\hat G$) must vanish on at least one of these characters, giving (iii).  Conversely, if (iii) holds, then we have $f * \chi = 0$ for some character $\chi \in \hat G$ which is of finite order, and hence periodic.  We arbitrarily select a retraction homomorphism $\psi \colon \C \to \Q$ from the complex numbers to the rationals.  Applying this retraction to $f * \chi = 0$, we obtain $f * (\psi \circ \chi) = 0$.  Since $\psi \circ \chi$ lies in $\ell^\infty(G,\Q)_{\p}$, by clearing denominators we have $M (\psi \circ \chi) \in \ell^\infty(G,\Z)_{\p}$ for some positive integer $M$.  Since $\psi \circ \chi$ is non-zero at the origin, we obtain (ii) by setting $a_{\p} \coloneqq M (\psi \circ \chi)$.

It remains to show that (i) implies (ii).  Without loss of generality we can write $G = \Z^d \times H$ for some dimension $d \geq 0$ and finite group $H$; we view $\Z^d$ as a subgroup of $G$.  The case $d=0$ is trivial, since $\ell^\infty(G,\Z) = \ell^\infty(G,\Z)_{\p}$ in this case, so now we assume inductively that $d \geq 1$ and that the claim has already been proven for groups of dimension $d-1$.

Suppose that $f*a = 0$ for some $a \in \ell^\infty(G,\Z) \backslash \{0\}$ for which $\partial_w a = 0$ for some non-zero $w \in \Z^d$.  Then $a$ is invariant on cosets of the cyclic group $\langle w \rangle$ generated by $v$, and thus descends to a non-zero element $\tilde a \in \ell^\infty(G/\langle w \rangle, \Z)$ on the quotient group $G/\langle w \rangle$, which has dimension $d-1$.  We can define the pushforward $\tilde f \in \ell^\infty(G/\langle w \rangle, \Z)_{\c}$ of $f$ by the formula
$$ \tilde f(x + \langle w \rangle) \coloneqq \sum_{y \in x + \langle w \rangle} f(y);$$
the identity $f*a=0$ then pushes forward to $\tilde f * \tilde a = 0$.  By induction hypothesis, we can then find a non-zero $\tilde a_{\p} \in \ell^\infty(G/\langle w \rangle, \Z)_{\p}$ such that $\tilde f * \tilde a_{\p} = 0$.  The non-zero periodic function $\tilde a_{\p}$ then pulls back to a non-zero periodic function $a_{\p} \in \ell^\infty(G, \Z)_{\p}$, and the identity $\tilde f * \tilde a_{\p} = 0$ pulls back to $f * a_{\p} = 0$, giving the desired claim (ii) in this case.

More generally, suppose that  $f*a = 0$ for some $a \in \ell^\infty(G,\Z) \backslash \{0\}$ for which $\partial_{w_1} \dots \partial_{w_m} a = 0$ for some non-zero $w_1,\dots,w_m \in \Z^d$.  Then there must exist $0 \leq i < m$ such that $\partial_{w_1} \dots \partial_{w_i} a$ is not identically zero, but $\partial_{w_1} \dots \partial_{w_{i+1}}  a$ is.  This now places us in the previous situation (with $a$ replaced by $\partial_{w_1} \dots \partial_{w_i} a$, and $w$ replaced by $w_{i+1}$), so we again have (ii) in this case.

The last remaining case is when $f*a = 0$ for some $a \in \ell^\infty(G,\Z)$ with the property that $\partial_{w_1} \dots \partial_{w_m} a$ does not vanish identically for any non-zero $w_1,\dots,w_m \in \Z^d$.  We now split $f$ as $f = \sum_{v \in V} f_v$ as in Theorem \ref{structure-thm}.  For every $v_0 \in V$, this theorem gives us a decomposition
$$ f_{v_0} * a = g - \sum_{w \in W_{v_0}} \varphi_{v_0,w}.$$
Each of the terms in the right-hand side is periodic in some direction, and thus annihilated by some $\partial_w$ for some non-zero $w \in \Z^d$.  Collecting all these directions together, we may find some finite collection $w_1,\dots,w_m \in \Z^d$ of non-zero vectors such that
$$ \partial_{w_1} \dots \partial_{w_m} (f_{v_0} * a) = 0$$
for all $v_0 \in V$.  Setting $\tilde a \coloneqq \partial_{w_1} \dots \partial_{w_m} a$, we conclude that
$$ f_{v_0} * \tilde a = 0$$
for all $v_0 \in V$.

By hypothesis, $\tilde a$ is not identically zero, thus $\tilde a$ is not identically zero on $\{w_*\} \times H$ for some $w_* \in \Z^d$.  This set is a fundamental domain for $G/\Z^d$.  If we then define
$$ a_\p \coloneqq \tilde a \circ \Pi_{\{w_*\} \times H, \Z^d} $$
then $a_{\p}$ is a $\Z^d$-periodic element of $\ell^\infty(G,\Z)_{\p} \backslash \{0\}$ with $f_{v_0} * a_\p$ agreeing with $f_{v_0} * \tilde a = 0$ on $v_0 + (\{w_*\} \times H)$, and hence on all of $\Z^d \times H$ by $\Z^d$-periodicity, for all $v_0 \in V$.  Summing in $v_0$, we conclude that $f*a_{\p}=0$, and so we obtain (ii) as desired.

\section{Solvability of $f*a=0$ in the integers is decidable}\label{sec:decidability}

In this section we use Theorem \ref{main-1} to prove Corollary \ref{cor:decidability-1}. 

We first need some preliminaries on roots of unity, and specifically when a finite number of such roots sum to zero.  For any $\xi \in \Q/\Z$, let $e(\xi)$ denote the root of unity $e(\xi) \coloneqq e^{2\pi i \xi}$.  For $(\xi_1,\dots,\xi_n) \in (\Q/\Z)^n$, we let $P_n(\xi_1,\dots,\xi_n)$ denote the predicate
\begin{equation}\label{eq:sum0}
        e(\xi_1)+\dots+e(\xi_n)=0.
\end{equation}
Thus for instance $P_n( \frac{0}{n} \mod 1, \frac{1}{n} \mod 1, \dots, \frac{n-1}{n} \mod 1)$ holds.  We define a \emph{rotation} of a tuple $(\xi_1,\dots,\xi_n) \in (\Q/\Z)^n$ to be a tuple of the form $(\xi_1+\xi,\dots,\xi_n+\xi)$ for some $\xi \in \Q/\Z$.  Clearly, $P_n$ is rotation-invariant: if a tuple $(\xi_1,\dots,\xi_n)$ obeys $P_n$, then so do all rotations of $(\xi_1,\dots,\xi_n)$.

A solution $(\xi_1,\dots,\xi_n) \in (\Q/\Z)^n$ to \eqref{eq:sum0} will be called \emph{minimal} if there is no proper subcollection of the $e(\xi_1),\dots,e(\xi_n)$ that sum to zero, thus $\sum_{j \in I} e(\xi_j) \neq 0$ for all proper subsets $\emptyset \subsetneq I \subsetneq \{1,\dots,n\}$ of $\{1,\dots,n\}$.  We let $P_{n,\min}(\xi_1,\dots,\xi_n)$ denote the assertion that $(\xi_1,\dots,\xi_n)$ is a minimal solution to \eqref{eq:sum0}.  This predicate is also clearly rotation-invariant.  We recall the following structural result about minimal solutions:

\begin{lemma}\label{ll-lem} If $P_{n,\min}(\xi_1,\dots,\xi_n)$ holds, then after applying a rotation, there exists a squarefree number $M$ such that $(\xi_1,\dots,\xi_n)$ has order dividing $M$ (i.e., $M\xi_j=0$ for all $1 \leq j\leq n$, or equivalently the $e(\xi_j)$ are all $M^{\mathrm{th}}$ roots of unity).  In fact one can take $M$ to be the product of all the primes less than or equal to $n$.
\end{lemma}

\begin{proof} This follows from \cite[Theorem 1]{mann}\footnote{In \cite[Theorem 5]{conway-jones} an improved bound on $M$ was found: for any $\xi_1,\dots,\xi_n$, one could select $M$ so that $\sum_{p|M} (p-2) \leq n-2$. As noted in that paper, this allows $M$ to be of size $\exp(O(n^{1/2} \log n))$ rather than $\exp(O(n))$. However, for our purposes the only thing that is important is that $M$ can be bounded by an explicitly computable quantity depending only on $n$.}.
Alternate proofs of the first claim may be found in \cite[Theorem 1]{conway-jones} or \cite[Corollary 3.2]{LL}.
\end{proof}

\begin{corollary}\label{chi-decide}  Let $n$ be  natural number.  Then the predicates $P_n(\xi_1,\dots,\xi_n)$ and $P_{n,\min}(\xi_1,\dots,\xi_n)$ are expressible (in a computable fashion) in the first-order language of $\Q/\Z$ as an Abelian group (in which the elements of $\Q/\Z$ are constants, and the operations are $+$ and $-$).
\end{corollary}

\begin{proof}  A tuple $(\xi_1,\dots,\xi_n)$ solves \eqref{eq:sum0} if and only if it can be partitioned into minimal tuples.  The number of ways to perform such a partition is clearly finite for a given $n$. Thus, in order to show that the predicates $P_n(\xi_1,\dots,\xi_n)$ are expressible in the first-order language of $\Q/\Z$, it suffices to do so for $P_{n,\min}(\xi_1,\dots,\xi_n)$.

Let $M$ be the product of all the primes less than or equal to $n$, and let $\Omega_n$ be the set of all tuples $(\xi_1,\dots,\xi_n)$ obeying $P_{n,\min}$ that are of order dividing $M$.  Clearly $\Omega_n$ is a computable finite set.  By Lemma \ref{ll-lem}, $P_{n,\min}(\xi_1,\dots,\xi_n)$ holds if and only if there exists $\xi \in \Q/\Z$ such that $(\xi_1+\xi,\dots,\xi_n+\xi) \in \Omega_n$.  This is clearly a (computable) predicate in the first-order language of $\Q/\Z$.
\end{proof}

\begin{remark} This corollary can also be obtained from \cite[Theorem 3]{conway-jones} (which, in the words of the authors, ``completes the proof... that trigonometric diophantine equations reduce to ordinary ones''). Explicit descriptions of $\Omega_n$ (and thus $P_{n,\min}$ and $P_n$) are known for some values of $n$; for instance, the cases $n \leq 7$ are treated in \cite{mann}, $n \leq 9$ in \cite{conway-jones}, and $n=12$ in  \cite[Section 3]{PR}.  On the other hand, the structure of minimal sums can get quite complicated when $n$ has many prime factors; see e.g., \cite{Steinberger}.
\end{remark}

The significance of this fact for the purpose of decidability arises from the following result of Wanda Szmielew \cite{Szmielew}.

\begin{theorem}\label{sz} The first-order theory of $\Q/\Z$ is decidable: there is a procedure which determines, in finite time, whether any given first-order sentence in this language is true.
\end{theorem}

\begin{proof}  This follows from Szmielew's theorem\footnote{We acknowledge ChatGPT for the assistance in finding this reference.} \cite{Szmielew} on the decidability of the theory of divisible Abelian groups.
\end{proof}

Combining this with Theorem \ref{main-1} and Corollary \ref{chi-decide}, we can now prove Corollary \ref{cor:decidability-1}. 

\begin{proof}[Proof of Corollary \ref{cor:decidability-1}]  By the classification of finitely generated Abelian groups, we can express $G$ explicitly as
$$ G = \Z^d \times \prod_{m=d+1}^r \Z/N_m \Z$$
for some $0 \leq d \leq r$ and some $N_{d+1},\dots,N_r \geq 1$.

Fix $f \in \ell^\infty(G,\Z)_{\c}$.  By Theorem \ref{main-1}, there is a non-zero $a\in \ell^\infty(G,\Z)$ such that $f*a=0$ if and only if there exists $\chi\in \hat G$ of finite order such that $\hat{f}(\chi)=0$.  By Theorem \ref{sz}, it then suffices to show that this latter claim can be expressed (in a computable fashion) as a first-order sentence in the language of $\Q/\Z$.

We express $f$ as a signed sum of Kronecker delta functions
$$ f = \sum_{j=1}^n e(\epsilon_j) \one_{\{g_j\}},$$
where $\epsilon_j \in \{0 \mod 1,1/2 \mod 1\}$ and $g_j \in G$ for $j=1,\dots,n$.  Thus, for any character $\chi$,
$$ \hat f(\chi) = \sum_{j=1}^n e(\epsilon_j) \overline{\chi(g_j)}.$$
When $\chi$ is finite order, the $\chi(g_j)$ are roots of unity, although not completely unconstrained due to relations between the $g_j$, as well as torsion in the group $G$.
Nevertheless, we may conclude that the assertion that there exists $\chi\in \hat G$ of finite order such that   $\hat{f}(\chi)=0$ is equivalent to the predicate
$$ \exists \xi_1,\dots,\xi_n \in (\Q/\Z)^n: P_n(\xi_1,\dots,\xi_n) \wedge Q_n(\xi_1,\dots,\xi_n),$$
where $P_n$ is the predicate defined above, and $Q_n(\xi_1,\dots,\xi_n)$ is the predicate that there exists $\chi \in \hat G$ such that $e(\xi_j) = e(\epsilon_j) \overline{\chi(g_j)}$ for all $j=1,\dots,n$.  In view of Corollary \ref{chi-decide}, it suffices to show that $Q_n(\xi_1,\dots,\xi_n)$ is equivalent to a (computable) predicate in the first-order language of $\Q/\Z$.

Let $\alpha_1,\dots,  \alpha_r$ be the standard generators of $G= \Z^d \times \prod_{m=d+1}^r \Z/N_m \Z$, where, in particular, $\alpha_1,\dots,\alpha_d$ is the standard basis of $\Z^d$.  Then we can express each of the $g_j$, $j=1,\dots,n$, in terms of these generators as
$$ g_j =\sum_{m=1}^r b_{j,m}\alpha_m$$
for some integers $b_{j,m}$.   Then we have
$$ \overline{\chi(g_j)} = e(- \sum_{m=1}^r b_{j,m} \eta_m )$$
if $\eta_m$ is the unique element of $\Q/\Z$ such that $\chi(\alpha_m) = e(\eta_m)$ for all $m=1,\dots,r$.  For $d+1 \leq m\leq r$, we have $\chi(N_m \alpha_m)=1$ and hence $N_m \eta_m = 0$.  Conversely, if $(\eta_m)_{m=1,\dots,r} \in (\Q/\Z)^r$ is such that $N_m \eta_m = 0$ for all $d+1 \leq m \leq r$, then there exists a finite order character $\chi$ such that $\chi(\alpha_m) = e(\eta_m)$ for all $m=1,\dots,r$.  We conclude that the predicate $Q_n(\xi_1,\dots,\xi_n)$ is equivalent to the predicate
$$ \exists \eta_1,\dots,\eta_r \in \Q/\Z: R(\xi_1,\dots,\xi_n,\eta_1,\dots,\eta_r) \wedge S(\eta_1,\dots,\eta_r)$$
where $R(\xi_1,\dots,\xi_n,\eta_1,\dots,\eta_r)$ is the assertion that
$$ \xi_j = \epsilon_j -   \sum_{m=1}^r b_{j,m} \eta_m$$
for all $1 \leq j \leq n$, and $S(\eta_1,\dots,\eta_r)$ is the assertion that
$$ N_m \eta_m = 0$$
for all $d+1 \leq m \leq r$.  This is clearly a computably expressible predicate in the first-order language of $\Q/\Z$ as an Abelian group, as desired. Thus, by Theorem \ref{sz}, one can algorithmically determine whether there is a non-zero $a\in \ell^\infty(G,\Z)$  such that $f*a=0$.
\end{proof}

\section{Integer-valued PTC in \texorpdfstring{$\Z^2$}{Z^2}}\label{integer-ptc}

We now prove Theorem \ref{main-2}.  Let $f,g$ be as in that theorem.  We may assume that $f$ is not identically zero, since the claim is trivial otherwise.

By hypothesis and Theorem \ref{structure-thm-2}(i), we have a structured solution $$(a, \tilde g, W, (\varphi_w)_{w \in W})$$ with $\tilde g \in \ell^\infty(\Z^2,\Q)_{\p}$.

Let $q$ be sufficiently divisible. The key issue here is that each of the $\varphi_w$ is only periodic in one direction $qw$, whereas we need periodicity in two directions to conclude the theorem.  In order to gain the extra periodicity we first modify the decomposition to be more ``rational''.  We arbitrarily select a retraction homomorphism $\psi \colon \R \to \Q$ from the reals to the rationals.  Since $a$ and $\tilde g$ already take values in $\Q$, we can thus apply $\psi$ to \eqref{a-rep} and conclude that
\begin{equation}\label{apw}
 a = \tilde g - \sum_{w \in W} \psi \circ \varphi_w.
 \end{equation}
The retraction homomorphism $\psi$ cannot possibly be continuous (or even measurable), so $\psi \circ \varphi_w$ will most likely be unbounded (so in particular would not lie in any of the spaces in Figure \ref{fig:inclusions}).  Nevertheless, we can still manipulate the $\psi \circ \varphi_w$ with convolutions by elements of $\ell^\infty(\Z^2,\Z)_{\c}$, and in particular with difference operators $\partial_h$.  For instance, since $\varphi_w$ is $qw$-periodic for sufficiently divisible $q$, $\psi \circ \varphi_w$ is also, thus $\psi \circ \varphi_w$ is annihilated by $\partial_{qw}$.  

Now let $w_0 \in W$.  Taking \eqref{apw} modulo one to eliminate the integer-valued function $a$, we see that
$$ \psi \circ \varphi_{w_0} = \tilde g - \sum_{w \in W \backslash \{w_0\}} \psi \circ \varphi_w \mod 1.$$
Every term on the right-hand side may be annihilated by some difference operator $\partial_h$ with $h \in \Z^2 \backslash \langle w_0\rangle$.  We conclude an identity of the form
\begin{equation}\label{phh} \partial_{h_1} \dots \partial_{h_l} (\psi \circ \varphi_{w_0}) = 0 \mod 1
\end{equation}
for some $h_1,\dots,h_l \in \Z^2 \backslash \langle w_0\rangle$ (depending on $w_0$).  As in \cite[Section 4]{GT21} (or \cite{BH}), this implies that $\psi \circ \varphi_{w_0}$ is piecewise polynomial modulo $1$; indeed, recalling that $w^*_0$ is a complementary vector to $w_0$, then (making $q$ more divisible as necessary), on each coset $x+q\Z^2$ 
the function
$$ (n,m) \mapsto \psi \circ \varphi_{w_0}( x + n qw_0 + m qw^*_0 ) \mod 1$$
for $n,m \in \Z$ is independent of $n$, and is annihilated by $\partial_{(0,1)}^l$, hence is equal to a polynomial $P_x(m)$ of degree at most $l-1$:
$$ \psi \circ \varphi_{w_0}( x + n qw_0 + m qw^*_0 ) = P_x(m) \mod 1.$$
The key point is that the retraction $\psi$ ensures that $\psi \circ \varphi_{w_0}$ takes only rational values, hence $P_x$ does also.  By the Lagrange interpolation formula, this means that all the coefficients of the polynomial $P_x$ are rational, which implies that $P_x$ is also periodic.  This implies that $\psi \circ \varphi_{w_0} \mod 1$ is periodic as well.  By making $q$ more divisible if necessary, we may assume that
$\psi \circ \varphi_{w_0} \mod 1$ is $q\Z^2$-periodic.

On the other hand, from Theorem \ref{structure-thm-2}(ii) we also know (making $q$ more divisible as needed) that for each coset $x + \langle w_0\rangle$ of $w_0$, the function
$(\one_{x + \langle w_0 \rangle} f) * \varphi_{w_0}$ is $q\Z^2$-periodic.  Since
$\one_{x + \langle w_0 \rangle} f$ is integer-valued, we may apply $\psi$ and conclude that $(\one_{x + \langle w_0 \rangle} f) * (\psi \circ \varphi_{w_0})$ is also $q\Z^2$-periodic.

We are now finally in a position to locate a periodic ``replacement'' for $\varphi_{w_0}$ for each $w_0 \in W$.  Namely, we define the function $\varphi_{w_0,\p} \in \ell^\infty(\Z^2,\Q)_{\p}$ by the formula
$$
\varphi_{w_0,\p} = \psi \circ \varphi_{w_0}\circ \Pi_{\Omega,q\Z^2},$$
where $\Omega$ is the fundamental domain for $\Z^2/q\Z^2$ defined by
$$ \Omega \coloneqq \{ aw_0 + bw^*_0: a,b \in [q] \}.$$
Clearly $\varphi_{w_0,\p}$ is $q\Z^2$-periodic; it agrees with $\psi \circ \varphi_{w_0,\p}$ not only on $\Omega$, but on the larger domain $\Omega' \coloneqq \{ aw_0 + bw^*_0: a \in \Z, b \in [q]\}$, because both functions are $qw_0$-periodic.  The identity
\begin{equation}\label{varphip}
 \varphi_{w_0,\p} = \psi \circ \varphi_{w_0} \mod 1
 \end{equation}
holds by construction at any point of $\Omega$ (or $\Omega'$), and hence on all of $\Z^2$ since both sides of the identity are $q\Z^2$-periodic modulo $1$.  In a similar vein, the identity
$$ (\one_{x + \langle w_0 \rangle} f) * \varphi_{w_0,\p} = (\one_{x + \langle w_0 \rangle} f) * (\psi \circ \varphi_{w_0})$$
holds by construction at any point of the form $x + \Omega'$, and hence on all of $\Z^2$ since both sides are $q\Z^2$-periodic. Summing in $x$, we conclude that
\begin{equation}\label{fvar}
f * \varphi_{w_0,\p} = f * (\psi \circ \varphi_{w_0}).
\end{equation}
Now set
\begin{equation}\label{ap-def}
a_\p \coloneqq \tilde g - \sum_{w \in W} \varphi_{w,\p}.
\end{equation}
Since $\tilde g$ and the $\varphi_{w,\p}$ are periodic, $a_\p$ is also.  From
\eqref{apw} and \eqref{varphip} we see that
$$ a_\p = a \mod 1;$$
thus, since $a$ is integer-valued, we conclude that $a_\p$ is also, which gives $a_\p \in \ell^\infty(\Z^2,\Z)_{\p}$.  By convolving \eqref{apw}, \eqref{ap-def} with $f$ and using \eqref{fvar}, we also have
$$ f * a_\p = f*a = g,$$
and the claim follows.

\section{Transforming real solutions to rational solutions}\label{rationalization-sec}

In this section we investigate questions of the following model-theoretic nature: if a given predicate $P(x_1,\dots,x_d)$ in $d$ variables can be satisfied with some real choice of variables $x_1,\dots,x_d \in \R$, can it also be satisfied with \emph{rational} choices $x_1, \dots, x_d \in \Q$ of variables?  We will informally refer to predicates with the above property as ``rationalizable'' predicates.  The property of being rationalizable will turn out to be central to the resolution of the periodic tiling conjecture at higher level in $\Z^2$ (Theorem \ref{main-3}), as the tilings we will consider will involve certain real parameters, and can be made periodic if those parameters are replaced with rational numbers which ``obey the same constraints'' as the original real parameters.\footnote{A similar rationalization procedure  also arises in measurable tilings of the torus: see \cite[Section 5.4]{ggrt}.}

The answer to this question will, of course, depend on the predicate; for instance, the nonlinear predicate $x^2=2$ famously admits real solutions but no rational solutions, and is thus not rationalizable. In the case where the predicate $P$ consists of linear constraints over the rationals, an affirmative answer can easily be obtained via retraction homomorphisms:

\begin{proposition}[Linear equations are rationalizable]  Suppose one has a system
\begin{equation}\label{system}
 a_{1,j} x_1 + \dots + a_{d,j} x_d = b_j
 \end{equation}
of linear equations in $d$ real unknowns $x_1,\dots,x_d$ where $j$ ranges over a (possibly infinite) index set, and the $a_{i,j}, b_j$ are all rational numbers.  If there exists a real solution $x_1,\dots,x_d \in \R$ to \eqref{system}, then there also exists a 
rational solution $x^\Q_1,\dots,x^\Q_d \in \Q$.
\end{proposition}

\begin{proof}  Select a retraction homomorphism $\psi \colon \R \to \Q$.  Applying $\psi$ to \eqref{system} we obtain
$$  a_{1,j} \psi(x_1) + \dots + a_{d,j} \psi(x_d) = b_j
$$
and the claim follows by setting $x_i^\Q \coloneqq \psi(x_i)$.
\end{proof}

In our proof of Theorem \ref{main-3}, we will need to deal with systems of linear \emph{inequalities} rather than equalities.  As retraction homomorphisms cannot be order-preserving, the preceding approach no longer applies to this setting. Nevertheless, one still has rationalizability as long as only finitely many inequalities are involved.  We first introduce some notation:

\begin{definition}[Language of linear inequalities over $\Q$]  We let ${\mathcal L}$ denote the (zeroth-order) language of linear inequalities with rational coefficients.  That is to say, given some finite number of indeterminates $x_1,\dots,x_d$ (which in practice will be either real-valued or rational-valued), a formula in ${\mathcal L}$ is a finite boolean combination of primitive formulae of the form
\begin{equation}\label{axb}
a_1 x_1 + \dots + a_d x_d \ ? \ b,
\end{equation}
where $? \in \{=, \neq, <, >, \leq, \geq\}$ and $a_1,\dots,a_d, b$ are rational numbers.  By abuse of notation, we identify a formula $P$ in ${\mathcal L}$ with the predicate $P(x_1,\dots,x_d)$ it induces on $d$ real (or rational) variables $x_1,\dots,x_d$.

A predicate $P(x_1,\dots,x_d)$ is said to be \emph{locally definable in ${\mathcal L}$} if, for every compact subset $K$ of $\R^n$, there exists a formula $S_K$ in ${\mathcal L}$ such that the solution set of $P$ agrees with that of $S_K$ on $K$, i.e.,
$$ \{ (x_1,\dots,x_d) \in K: P(x_1,\dots,x_d) \} = \{ (x_1,\dots,x_d) \in K: S_K(x_1,\dots,x_d) \}.$$
\end{definition}

\begin{examples}\   
\begin{itemize}
 \item[(i)] The predicate
$$P(x_1,x_2)\colon x_1 - \frac{1}{2} x_2 \in \{0,1\}$$
is described by a formula in ${\mathcal L}$, namely
$$ \left(x_1 - \frac{1}{2} x_2 = 0\right) \vee \left(x_1 - \frac{1}{2} x_2 = 1\right).$$
  \item[(ii)] The predicate
$$ P(x_1,x_2)\colon (x_1,x_2) \in [0,1)^2$$
is described by a formula in ${\mathcal L}$, namely
$$ (x_1 \geq 0) \wedge (x_1 < 1) \wedge (x_2 \geq 0) \wedge (x_2 < 1).$$
\item[(iii)] The predicate
$$P(x)\colon x^2 = 2$$
\emph{cannot} be described by as a formula in ${\mathcal L}$; this follows from Proposition \ref{lin-ineq} below, as well as the irrationality of $\pm \sqrt{2}$.
\item[(iv)]  The predicate
$$P(x)\colon \{x\} < 0.5$$
is not globally definable in ${\mathcal L}$ (it requires an infinite number of linear inequalities to describe completely).  However, it is locally definable in ${\mathcal L}$.  For instance, on the compact set $[0,2]$, the solution set of the  predicate $P(x)$ agrees with the solution set of the sentence
$$S_{[0,2]}\colon ((x \geq 0) \wedge (x < 0.5)) \vee ((x \geq 1) \wedge (x < 1.5)).$$
\end{itemize}
\end{examples}

\begin{proposition}[Linear inequalities are rationalizable]\label{lin-ineq}  Let $P(x_1,\dots,x_d)$ be a predicate that is locally definable in ${\mathcal L}$.  If there exists a real solution $x^\R_1,\dots,x^\R_d \in \R$ to this predicate, then there also exists a rational solution $x^\Q_1,\dots,x^\Q_d \in \Q$.
\end{proposition}

Heuristically: finite linear programming over the rationals is incapable of producing only irrational solutions. The requirement of only finitely many inequalities is necessary, since, e.g., one can describe a given irrational number such as $\sqrt{2}$ as the unique solution to a system of an infinite number of inequalities $b_n^- < x < b_n^+$, where $b_n^+, b_n^-$ are sequences of rational numbers converging to $\sqrt{2}$ from above and below respectively.

\begin{proof}[Proof of Proposition \ref{lin-ineq}]  Let $U$ be a precompact open neighborhood of the point $(x^\R_1,\dots,x^\R_d)$ in $\R^d$.  The solution set $\{ (x_1,\dots,x_d) \in \R^n: P(x_1,\dots,x_d)\}$ to the given predicate agrees on $U$ with a boolean combination of finitely many hyperplanes and half-spaces defined over the rationals.  
Let ${\mathcal F}$ denote the class of finite unions of open convex polytopes in various affine subspaces  of $\R^d$ defined over the rationals.  This space contains all hyperplanes and half-spaces defined over the rationals, and is closed under boolean operations (with the key point being that the complement of an open convex polytope is a finite union of open convex polytopes); hence the solution set of $P$ agrees on $U$ with an element of ${\mathcal F}$. Since this solution set contains $(x^\R_1,\dots,x^\R_d)$, this set must therefore contain at least one non-empty open convex polytope $Q$ in some $d'$-dimensional affine subspace $V$ of $\R^d$ defined over the rationals, for some $0 \leq d' \leq d$.  One can view $V$ as the image $V = T(\R^{d'})$ of a standard Euclidean space $\R^{d'}$ by some linear injective map $T \colon \R^{d'} \to \R^d$ with rational coefficients, then $Q = T(Q')$  for some non-empty open convex polytope $Q'$ in $\R^{d'}$.  Since $\Q^{d'}$ is dense in $\R^{d'}$ and $U$ is an open neighborhood of $(x_1,\dots,x_d)$, we may find a rational point $q$ in $Q'$ with $T(q)$ lying in $U$, and then $T(q)$ will be a rational solution to $P$, as desired.
\end{proof}

\begin{remark}[Quantifier elimination]\label{lin-program} Thanks to the algorithms of linear programming (or the fact that the linear projection of a polytope is again a polytope), the language of linear inequalities over the rationals admits quantifier elimination over the reals.  Thus, every first-order predicate in ${\mathcal L}$ over the reals is in fact equivalent to a zeroth-order predicate. For instance, the predicate
$$ \exists y: (x+y < 1) \wedge (y-x < 1) \wedge (y > 0)$$
is equivalent to the predicate
$$ (x > -1) \wedge (x < 1).$$
As such, one can extend the language ${\mathcal L}$ in Proposition \ref{lin-ineq} from zeroth order sentences to first order sentences.
\end{remark}

Let $d \geq 0$.  For any field $k$ containing the rationals, we let $\Hom_1(\Q^{d+1},k)$ denote the set of all $\Q$-linear maps $\Phi \colon \Q^{d+1} \to k$ with $\Phi(e_0)=1$, where $e_0,\dots,e_d$ is the standard basis for $\Q^{d+1}$; this space can be identified with $k^d$ by identifying $\Phi$ with $(\Phi(e_1),\dots,\Phi(e_d))$.  In particular, $\Hom_1(\Q^{d+1},\Q)$ is dense in $\Hom_1(\Q^{d+1},\R)$.  With these identifications, a linear equality or inequality of the form \eqref{axb} can then be written more compactly as
\begin{equation}\label{axb-pred} \Phi(a) \ ? \ b
\end{equation}
where $\Phi \in \Hom_1(\Q^{d+1},\R)$ is the unknown, and $a \in \Q^{d+1}$ is the vector\footnote{One could also normalize $b$ by zero by replacing the first component of $a$ with $-b$, if desired.} $a \coloneqq (0,a_1,\dots,a_{d})$.   Proposition \ref{lin-ineq} can then be restated as follows: if a predicate $P(\Phi)$ that is locally definable in ${\mathcal L}$ has a solution $\Phi$ in $\Hom_1(\Q^{d+1},\R)$, then it also has a solution $\Phi^\Q$ in $\Hom_1(\Q^{d+1},\Q)$.

For technical reasons (relating to the desire to avoid the singularities of the fractional part operation $\{\}$) we will need a more complicated version of this assertion in which the initial solution $\Phi$ is assumed to be injective, and as a consequence some additional non-degeneracy properties can be placed on the rationalized solution $\Phi^\Q$.

\begin{proposition}[Linear inequalities with a non-degeneracy condition are realizable]\label{nondeg-real}  Let $d \geq 1$ and $J \geq 1$ be natural  numbers, and for $j=1,\dots,J$, let  $\overline{\alpha}_j, \overline{\beta}_j \in \Q^{d+1}$ be such that $\overline{\alpha}_j, \overline{\beta}_j, e_0$ are linearly independent over $\Q$.  Let $P(\Phi)$ be a predicate that is locally definable in ${\mathcal L}$ (identifying $\Hom_1(\Q^{d+1},\R)$ with $\R^d$).  If $P$ has an injective solution $\Phi \in \Hom_1(\Q^{d+1},\R)$, then $P$ also has a rational solution $\Phi^\Q \in \Hom_1(\Q^{d+1},\Q)$ with the additional ``non-degeneracy'' property that for all $1 \leq j \leq J$, $\Phi^\Q(\overline{\beta_j})$ is not an integer linear combination of 
$\Phi^\Q(\overline{\alpha_j})$ and $1$.
\end{proposition}

The injectivity of $\Phi$ is necessary for this proposition, since otherwise one can easily concoct examples of predicates $P$ that force (say) $\Phi^\Q(\overline{\beta_1})$ to vanish. 

\begin{proof}[Proof of Proposition \ref{nondeg-real}]  Again we let $U$ be a precompact open neighborhood of $\Phi$.  In $U$, the solution set $\{ \Phi \in \Hom_1(\Q^{d+1},\R): P(\Phi) \}$ is the union of finitely many open convex polytopes in various subspaces of $\Hom(\Q^{d+1},\R)$ defined over the rationals.  One of these polytopes must contain the injective solution $\Phi$.  This injective solution cannot be contained an any proper affine subspace of $\Hom(\Q^{d+1},\R)$ defined over the rationals, as this would imply a non-trivial relation of the form
$$ a_1 \Phi(e_1) + \dots + a_d \Phi(e_d) = a_0 = a_0 \Phi(e_0)$$
for some rationals $a_0,\dots,a_d$, not all zero; but this contradicts the injectivity of $\Phi$, since the relation can be rewritten as
$$ \Phi(a_1 e_1 + \dots + a_d e_d) = \Phi(a_0 e_0).$$
Thus, the open convex polytope $Q$ that contains $\Phi$ must be open in the entire space $\Hom_1(\Q^{d+1},\R)$, rather than in a proper subspace of $\Hom_1(\Q^{d+1},\R)$.

Any rational point $\Phi^\Q \in \Hom_1(\Q^{d+1},\Q)$ in the non-empty open set $Q \cap U$ would be a rational solution to the predicate $P$; the remaining difficulty is to ensure the non-degeneracy condition that $\Phi^\Q(\overline{\beta_j})$ is not an integer linear combination of $\Phi^\Q(\overline{\alpha_j})$ and $1$.  

Because we are now working over the rationals instead of the reals, we cannot appeal to a Baire category argument or a measure-theoretic argument here; we turn instead to $p$-adic (or ``adelic'') methods.  Given a prime $p$ and a non-zero rational number $x=a/b$, define the $p$-valuation $\nu_p(x) \in \Z$ to be the number of times $p$ divides $a$, minus the number of times $p$ divides $b$.  We then define the $p$-adic metric $d_p$ on $\Q$ by the formula $d_p( x, y ) \coloneqq p^{-\nu_p(x-y)}$ for $x \neq y$ (and $d_p(x,x)=0$).  It is easy to see that this is a metric; we then define the $p$-adic field $\Q_p$ to be the metric completion of $\Q$ under $d_p$, and extend $\nu_p$ to $\Q_p$ in the usual fashion (with the convention $\nu_p(0)=+\infty$).  We define the ring of $p$-adic integers $\Z_p$ to be the metric completion of $\Z$ under the same metric; we observe that the metric balls in $\Q_p$ are of the form $x + p^n \Z_p$ for $x \in \Q_p$ and $n \in \Z$.

Let $p_1,\dots,p_J$ be distinct primes.  The key point is that if
$$ \nu_{p_j}( \Phi^\Q(\overline{\beta_j}) ) < \min(0, \nu_{p_j}( \Phi^\Q(\overline{\alpha_j})) ) $$
for some $j=1,\dots,J$, then $\Phi^\Q(\overline{\beta_j})$ cannot be an integer linear combination of $\nu_{p_j}( \Phi^\Q(\overline{\alpha_j}) )$ and $1$ (it is divisible by fewer powers of $p_j$).  If we then let 
$$ \Omega_j \coloneqq \{ \tilde \Phi \in \Hom_1(\Q^{d+1}, \Q_{p_j}) \colon \nu_{p_j}( \tilde \Phi(\overline{\beta_j}) ) < \min( 0, \nu_{p_j}( \tilde \Phi(\overline{\alpha_j}) )) \},$$
then $\Omega_j$ is an open subset of $\Hom(\Q^{d+1}, \Q_{p_j})$, and our task then reduces to finding a point $\Phi^\Q \in \Hom(\Q^{d+1},\Q)$ that simultaneously lies in $Q \cap U$ and in all of the $\Omega_j$, $1\leq j\leq J$.

Fix $j=1,\dots,J$.  Because $\overline{\beta_j}, \overline{\alpha_j}, e_0$ are linearly independent over $\Q$, we see that the quantities $\tilde \Phi(\overline{\beta_j})$ and $\tilde \Phi(\overline{\alpha_j})$ can vary independently in $\Q_{p_j}$ as $\tilde \Phi$ ranges over $\Hom_1(\Q^{d+1}, \Q_{p_j})$.  In particular, $\Omega_j$ is non-empty.  As $\Omega_j$ is open in the $p$-adic topology, it contains a set of the form
$$ q_j + p_j^{n_j} \Z_{p_j}^{d'}$$
for some $q_j \in \Q^{d'}$ and some $n_j \geq 0$.  By the Chinese remainder theorem, we can thus find $q_0 \in \Q^{d'}$ and a modulus $P\geq 1$ (the product of the $p_j^{n_j}$) with the property that
$$ q_0 + P \frac{a}{b} \in \Omega_j$$
for all $j=1,\dots,J$, whenever $a \in \Z^{d'}$ and $b \geq 1$  is a natural number such that $b$ is coprime to $p_1 \dots p_J$ (and thus invertible in $\Z_{p_j}$ for every $j$).  In particular, the elements of $\Q^{d'}$ that lie in all of the $\Omega_j$ will contain a coset of $\frac{P}{b}\Z^{d'}$ whenever $b$ is coprime to $p_1 \dots p_J$.  By taking $b$ large enough, this coset will intersect the open set $Q \cap U$ in some point $\Phi^\Q \in \Hom(\Q^{d+1},\Q)$, as desired.
\end{proof}

\section{Higher level PTC in \texorpdfstring{$\Z^2$}{Z^2}, I: cleaning the solution}\label{cleaning-sec}

In this section we begin the proof of Theorem \ref{main-3}.  

\subsection{Strategy}\label{strategy-sec}
Let $f, g$ be as in the theorem.  We may assume that $f$ is not identically zero, as the claim is trivial in that case. By Theorem \ref{structure-thm-2}(i), there exists an indicator function structured solution $(\one_A, \tilde g, W, (\varphi_w)_{w \in W})$ to the equation $f*a=g$.  The finite set $W$ of linearly independent primitive elements will be fixed henceforth, but we will consider modifications of the other components $\one_A$, $\tilde g$, $\varphi_w$ of this structured solution.  Indeed, to prove the theorem, it will suffice to find another indicator function structured solution $(\one_{A_\p}, \tilde g_\p, W, (\varphi_{w,\p})_{w \in W})$ in which the set $A'$ is periodic.  Our strategy for doing this can be summarized as follows:

\begin{enumerate}
    \item ``Clean up'' the structured solution $(\one_A, \tilde g, W, (\varphi_w)_{w \in W})$ as much as possible, to put it in a ``normal form''. (Here we will use some ergodic theory methods to assist with the cleaning process, allowing us to avoid ``bad'' events, such as encountering a density zero set or a discontinuity of $\{\}$, that only occur on null sets in the ergodic formalism.)
    \item Describe this (normalized) solution in terms of an (injective) map $\Phi \in \Hom_1(\Q^{d+1},\R)$, together with some auxiliary ``combinatorial'' data, obeying (an infinite number of) additional conditions in the language ${\mathcal L}$ of linear inequalities with rational coefficients.
    \item Eliminate the role of the combinatorial data, and reduce the number of conditions from infinite to finite.
    \item Apply Proposition \ref{nondeg-real} to replace the real-valued map $\Phi \in \Hom_1(\Q^{d+1},\R)$ by a (non-degenerate) rational-valued map $\Phi^\Q \in \Hom_1(\Q^{d+1},\Q)$, obeying the same set of conditions.
    \item Use this to generate the periodic structured solution $(\one_{A_\p}, \tilde g_\p, W, (\varphi_{w,\p})_{w \in W})$.
\end{enumerate}

In this section, we perform the ``cleaning'' stage of this strategy, deferring the later portions of the strategy to the next section.

\subsection{Analyzing the structure of the functions $\varphi_w$} We will exploit a simple observation: we have the freedom to subtract an arbitrary periodic function in $\ell^\infty(\Z^2,\R)_{\p}$ from any of the $\varphi_w$, as long as we also subtract the same function from $\tilde g$ (to maintain the equation \eqref{a-rep-ind}), since such a periodic function will also be $qw$-periodic for sufficiently divisible $q$.  
As a first application of this observation, we localize the range of each $\varphi_w$ to the unit interval $[0,1]$. 

\begin{proposition}[Restricting the range to the unit interval]\label{unit-loc}   Let $f \in \ell^\infty(\Z^2,\Z)_{\c} \backslash \{0\}$, $A\subset \Z^2$ and $g \in \ell^\infty(\Z^2,\Z)_{\p}$ be such that $f*\one_A=g$.  Then
 there exists an indicator function structured solution $(\one_A, \tilde g, W, (\varphi_w)_{w \in W})$  in which all the functions $\varphi_w$, $w\in W$, take values in the unit interval $[0,1]$.
\end{proposition}

\begin{proof}
    We begin with the indicator function structured solution $$(\one_A, \tilde g, W, (\varphi_w)_{w \in W})$$ provided by Theorem \ref{structure-thm-2}.  By the preceding observation, it will suffice to show that for each $w_0 \in W$, the function $\varphi_{w_0}$ differs from a $[0,1]$-valued function by a periodic function.
    
Fix $w_0 \in W$.  From \eqref{a-rep-ind} we have
$$ \varphi_{w_0} = - \one_A + \tilde g - \sum_{w \in W \backslash \{w_0\}} \varphi_w,$$
and thus on applying the projection operator $\pi_{qw_0}$ defined in \eqref{piv-def} for some sufficiently divisible $q$
$$ \varphi_{w_0} = - \pi_{qw_0} \one_A + \tilde g - \sum_{w \in W \backslash \{w_0\}} \pi_{qw_0} \varphi_w.$$
The function $\pi_{qw_0} \varphi_w$ is both $qw_0$-periodic and $qw$-periodic for sufficiently divisible $q$, and is thus periodic.  Meanwhile, from \eqref{contract} we see that $\pi_{qw_0} \one_A$ takes values in $[0,1]$, thus $\varphi_{w_0}$ differs from the $[0,1]$-valued function $1 - \pi_{qw_0} \one_A$ by a periodic function, as desired.
\end{proof}

Next, we refine the structure of the $\varphi_w$ a little more, in that these functions can be taken to be either constant, $\{0,1\}$-valued or essentially determined by equidistributed polynomials on each coset $x_0 + q\Z^2$ of a lattice $q\Z^2$.   For minor notational reasons, it will be convenient to normalize the base point $x_0$ of such cosets to lie in the fundamental domain $[q]^2$ for $\Z^2/q\Z^2$.

\begin{proposition}[Further structure on $\varphi_w$]\label{furt-struct} Let $(\one_A, \tilde g, W, (\varphi_w)_{w \in W})$ be an indicator function structured solution to $f * a = g$ in which all the functions $\varphi_w$, $w\in W$, take values in the unit interval $[0,1]$.  Then for all sufficiently divisible $q$, one can create a partition $W \times [q]^2 = \Gamma_{q,\cc} \cup \Gamma_{q,\e}$ with the following properties:
\begin{itemize}
    \item[(a)] (Combinatorial case) If $(w_0, x_0) \in \Gamma_{q,\cc}$, then $\varphi_{w_0}$ is either constant on $x_0+q\Z^2$, $\{0,1\}$-valued on $x_0+q\Z^2$, or both.
    \item[(b)] (Equidistributed case) If $(w_0, x_0) \in \Gamma_{q,\e}$, then there exists a non-constant polynomial $P_{w_0,q,x_0} \colon \Z \to \R$ with irrational leading coefficient, such that
\begin{equation}\label{equid}
\varphi_{w_0}( x_0 + n qw_0 + m qw^*_0 ) = P_{w_0,q,x_0}(m) \mod 1
\end{equation}
    for all $n,m \in \Z$.
\end{itemize}
\end{proposition}

Note from the irrationality of the leading coefficient\footnote{We focus here on the leading coefficient of $P_{w_0,q,x_0}$, rather than lower order coefficients, because this coefficient is stable under translations $m \mapsto m+h$.} of $P_{w_0,q,x_0}$ that the combinatorial and equidistributed cases are necessarily disjoint.  The polynomial $P_{w_0,q,x_0}$ is unique up to integer shifts in the coefficients of $P_{w_0,q,x_0}$.  We will later be able to lower the degree of these polynomials to be $1$, but for now we allow for the theoretical possibility that these polynomials are non-linear.

\begin{proof}[Proof of Proposition \ref{furt-struct}]
We first observe that it suffices to establish the claim for a single value of $q$, as this implies the same structure for larger multiples of $q$ (subdividing the cosets of $q\Z^2$ as needed).

As before, for any $w_0 \in W$ we have the relation
$$ \varphi_{w_0} = - \one_A + \tilde g - \sum_{w \in W \backslash \{w_0\}} \varphi_w$$
thanks to \eqref{a-rep-ind}.  Reducing modulo $1$ to eliminate $\one_A$, and applying the difference operators
$$  \partial_{h_1} \dots \partial_{h_l}, $$
for $h_1,\dots,h_l \in \Z^2 \backslash \langle w_0 \rangle$  as in \eqref{phh},  to eliminate the other terms on the right-hand side, we obtain an equation of the form
$$  \partial_{h_1} \dots \partial_{h_l} \varphi_{w_0} = 0 \mod 1 .$$
  As in Section \ref{integer-ptc}, for $q$ sufficiently divisible, this identity, together with the $qw_0$-periodicity of $\varphi_{w_0}$, implies the existence of a polynomial $P_{w_0, x_0} \colon \Z \to \R$ for each coset $x_0 + q\Z^2$ such that
$$ \varphi_{w_0}( x_0 + n qw_0 + m qw^*_0 ) = P_{w_0, x_0}(m) \mod 1$$
for all $n,m \in \Z$.

For those cosets $x_0+q\Z^2$ for which $P_{w_0, x_0}$ has at least one irrational non-constant coefficient, we are now in the equidistributed case (b) (after making $q$ more divisible as needed to make all coefficients higher than the highest irrational coefficient give constant terms) and are done in that case.  For the $x_0$ for which all the coefficients of $P_{w_0, x_0}$ (other than the constant term) are rational, the polynomial $P_{w_0, x_0}$ is periodic modulo $1$, and hence $\varphi_{w_0}$ is also.  Thus, by making $q$ more divisible if necessary, we can ensure that $\varphi_{w_0}$ is constant modulo $1$ (not just periodic) on each such coset.  Since $\varphi_{w_0}$ is also $[0,1]$-valued, this implies that either $\varphi_{w_0}$ is $\{0,1\}$-valued on $x_0+q\Z^2$, or else constant $c_{w_0, x_0}$ on $x_0+q\Z^2$.  In both cases we are in the combinatorial case (a).
\end{proof}

By a further modification of the structured solution, one can make the following technical improvements to the properties (a) and (b) in Proposition \ref{furt-struct}.

\begin{proposition}[Technical refinement]\label{furt-struct-2} Let $(\one_A, \tilde g, W, (\varphi_w)_{w \in W})$ be as in the Proposition \ref{furt-struct}.  Then one can find a further indicator function structured solution $(\one_{A'}, \tilde g, W, (\varphi'_w)_{w \in W})$ with the functions $\varphi'_w, w \in W$ taking values in the unit interval $[0,1]$, such that for all sufficiently divisible $q$, one can create a partition $W \times [q]^2 = \Gamma_{q,\cc} \cup \Gamma_{q,\e}$ with the following improved properties:
\begin{itemize}
    \item[(a')] (Combinatorial case) If $(w_0, x_0) \in \Gamma_{q,\cc}$, then either $\varphi'_{w_0}$ is constant on $x_0+q\Z^2$, or $\one_{A'} = 1 - \varphi'_{w_0}$ on $x_0+q\Z^2$ (so in particular $\varphi'_{w_0}$ is $\{0,1\}$-valued), or both.
    \item[(b')] (Equidistributed case) If $(w_0, x_0) \in \Gamma_{q,\e}$, then there exists a non-constant polynomial $P'_{w_0,q,x_0} \colon \Z \to \R$ with irrational leading coefficient, such that
\begin{equation}\label{equid'}
\varphi'_{w_0}( x_0 + n qw_0 + m qw^*_0 ) = P'_{w_0,q,x_0}(m) \mod 1
\end{equation}
    for all $n,m \in \Z$.  Furthermore, if $P'_{w_0,q,x_0}(m)$ is a linear polynomial $\alpha_{w_0,q,x_0} m + \beta_{w_0,q,x_0}$, then $\alpha_{w_0,q,x_0}, \beta_{w_0,q,x_0}, 1$ are linearly independent over $\Q$ (hence, by \eqref{equid'} and the $[0,1]$-valued nature of $\varphi'_{w_0}$, we can write in this case
\begin{equation}\label{equid''}
\varphi'_{w_0}( x_0 + n qw_0 + m qw^*_0 ) = \{ \alpha_{w_0,q,x_0} m + \beta_{w_0,q,x_0} \}.)
\end{equation}
\end{itemize}
\end{proposition}

\begin{proof}
We observe that it suffices to establish the proposition for a single sufficiently divisible $q$, as the same claim holds for multiples of $q$ by subdivision.  (Note that because we require the leading coefficient of $P'_{w_0,q,x_0}$ to be irrational, a nonlinear polynomial cannot become linear upon restricting to a finer subcoset.) 

We will use the probabilistic method, by constructing a probability space of structured solutions, and then showing that the conclusions hold almost surely (i.e., with probability one).
Let $(\one_A, \tilde g, W, (\varphi_w)_{w \in W})$ be the indicator function structured solution from Proposition \ref{furt-struct}.  Then there exists an integer $q_0$ such that $g$ and $\tilde g$ are $q_0\Z^2$-periodic, and for each $w \in W$ the function $\varphi_w$ is $[0,1]$-valued and $q_0w$-periodic.  We now allow the set $A$ and the functions $\varphi_w$ to vary, although we keep $q_0$, $\tilde g$ and $W$ fixed.  More precisely, let $X_{q_0,W, \tilde g}$ be the space of indicator function structured solutions $(\one_{A'}, \tilde g, W, (\varphi'_w)_{w \in W})$ to $f * a = g$ where each $\varphi'_w$ is $[0,1]$-valued and $q_0w$-periodic.  Then $X_{q_0,W, \tilde g}$ is a non-empty compact metrizable space using the topology of pointwise convergence, and also has a continuous $q_0\Z^2$-shift action by defining
$$ T^h (\one_{A'}, \tilde g, W, (\varphi'_w)_{w \in W}) \coloneqq (\one_{\{h\}} * \one_{A'}, \tilde g, W, (\one_{\{h\}} * \varphi'_w)_{w \in W})$$
for $h \in q_0\Z^2$ and $(\one_A, \tilde g, W, (\varphi'_w)_{w \in W}) \in X_{q_0,W, \tilde g}$.  By the Krylov--Bogolyubov theorem \cite{KB}, we can find a $q_0\Z^2$-invariant Borel probability measure\footnote{One could even make $\mathbb{P}$ $q_0\Z^2$-ergodic if desired, although we will not need to do so here.} $\mathbb{P}$ on $X_{q_0,W, \tilde g}$.

We view $(\one_{A'}, \tilde g, W, (\varphi'_w)_{w \in W})$ as a random variable taking values in $X_{q_0,W, \tilde g}$ with probability distribution $\mathbb{P}$.
Our task is now to show that for this random structured solution, the properties (a) and (b) from the Proposition \ref{furt-struct} can be upgraded to (a') and (b') respectively.

We begin with  the task of upgrading (a). Let $q$ be sufficiently divisible (in particular, $q_0$ divides $q$), let $v_0,v_1$ lie in the same coset of $q\Z^2$, and let $w_0 \in W$.  We claim: 
\begin{itemize}
    \item[(i)] If $(w_0,x_0) \in \Gamma_{q_0,\cc}$, $\varphi'_{w_0}(v_0)=0$ and $\varphi'_{w_0}(v_1)=1$, then almost surely $\one_{A'}(v_0) = 1$ and $\one_{A'}(v_1) = 0$.
\end{itemize}  
To see this, let $E$ denote the event that $\varphi'_{w_0}(v_0)=0$ and $\varphi'_{w_0}(v_1)=1$; this event is $qw_0$-invariant because $\varphi'_{w_0}$ is $qw_0$-periodic.  If $E$ holds, then for any natural number $n$, we have from \eqref{a-rep-ind} that
\begin{equation}\label{eq1}
    \one_{A'}( v_0 + nq_0w_0 ) = \tilde g(v_0+nq_0w_0) - \sum_{w \in W \backslash \{w_0\}} \varphi'_{w}(v_0 + n q_0 w_0)
\end{equation} 
and
\begin{equation}\label{eq2}
    \one_{A'}( v_1 + nq_0w_0 ) = \tilde g(v_1+nq_0w_0) - \sum_{w \in W \backslash \{w_0\}} \varphi'_{w}(v_1 + n q_0 w_0) - 1.
\end{equation} 
By the $q_0\Z^2$-periodicity of $\tilde g$, we have $\tilde g(v_0+nq_0w_0) = \tilde g(v_1+nq_0w_0)$. We can thus subtract equation \eqref{eq2} from \eqref{eq1} and conclude that
$$ \one_{A'}( v_0 + nq_0w_0 ) - \one_{A'}( v_1 + nq_0w_0 ) = 1 + \sum_{w \in W \backslash \{w_0\}} (\varphi'_{w}(v_1 + n q_0 w_0) - \varphi'_{w}(v_0 + n q_0 w_0)).$$
Since $v_1-v_0 \in q\Z^2$ and for each $w\in W\backslash\{w_0\}$ the function $\varphi'_w$ is $q_0w$-periodic and $w$ is linearly independent of $w_0$, we can write
$$ \varphi'_{w}(v_1 + n q_0 w_0) = \varphi'_{w}(v_0 + (n+h_w) q_0 w_0)$$
for some integer $h_w$ (which depends on $w_0, w, q, q_0, v_0, v_1$).  We can therefore average in $n$ and use telescoping series to conclude that
$$ \lim_{N \to \infty} \frac{1}{N} \sum_{n=1}^N (\one_{A'}( v_0 + nq_0w_0 ) - \one_{A'}( v_1 + nq_0w_0 )) = 1.$$
Since $\one_{A'}$ is bounded between $0$ and $1$, this implies by the triangle inequality that
$$ \lim_{N \to \infty} \frac{1}{N} \sum_{n=1}^N \one_{A'}( v_0 + nq_0w_0 ) = 1$$
and
$$ \lim_{N \to \infty} \frac{1}{N} \sum_{n=1}^N \one_{A'}( v_1 + nq_0w_0 ) = 0$$
whenever $E$ holds.  By dominated convergence, we conclude that
$$ \lim_{N \to \infty} \mathbb{E} \one_E \frac{1}{N} \sum_{n=1}^N \one_{A'}( v_0 + nq_0w_0 ) = 1$$
and
$$ \lim_{N \to \infty} \mathbb{E} \one_E \frac{1}{N} \sum_{n=1}^N \one_{A'}( v_1 + nq_0w_0 ) = 0.$$
By the $q_0w_0$-invariance of $\mathbb{P}$, we can simplify this to
$$ \mathbb{E} \one_E \one_{A'}( v_0 ) = 1 \quad \text{and} \quad  \mathbb{E} \one_E \one_{A'}( v_1  ) = 0.$$
The claim (i) follows.
As a consequence, we now have
\begin{itemize}
    \item[(ii)] If $(w_0,x_0) \in \Gamma_{q_0,\cc}$, and $\varphi'_{w_0}$ is $\{0,1\}$-valued, but not identically constant on $x_0+q\Z^2$, then almost surely $\one_{A'} = 1 - \varphi'_{w_0}$ on $x_0+q\Z^2$.
\end{itemize}
Indeed, by hypothesis we can find $v_0, v_1 \in x_0 + q\Z^2$ with $\varphi'_{w_0}(v_0)=0$ and $\varphi'_{w_0}(v_1)=1$.  Then almost surely, we may apply the previous claim (i) to various $v \in x_0+q\Z^2$ (combined with either $v_0$ or $v_1$) to conclude that $\one_{A'}(v)=1$ whenever $\varphi'_{w_0}(v)=0$, and $\one_{A'}(v)=0$ whenever $\varphi'_{w_0}(v)=1$, giving (ii).  
Thus we can almost surely upgrade the combinatorial case (a) to (a').

It remains to almost surely upgrade the equidistributed case (b) to (b') for a given $(w_0,x_0) \in \Gamma_{q_0,\e}$. As the countable union of null events is null, it suffices to show that:
\begin{itemize}
    \item[(iii)] If $(w_0,x_0) \in \Gamma_{q_0,\e}$, then for any integers $a,b,c$, not all zero, the event $E_{w_0,q_0,a,b,c}$ that (b) holds with $$P'_{w_0,q_0,x_0}(m) = \alpha_{w_0,q_0,x_0} m + \beta_{w_0,q_0,x_0}$$ obeying a linear relation $a \alpha_{w_0,q_0,x_0} + b \beta_{w_0,q_0,x_0} + c = 0$, is a null event. 
\end{itemize}
For $b=0$, this already follows from the irrationality of $\alpha_{w_0,q_0,x_0}$, so we may assume $b \neq 0$.  Observe that if $E_{w_0,q_0,a,b,c}$ holds, and we translate the structured solution
$(\one_{A'}, \tilde g, W, (\varphi'_w)_{w \in W})$ by $hq_0w^*_0$ for some integer $h$, then we stay in the equidistributed case but with $\beta_{w_0,q_0,x_0}$ shifted by $h \alpha_{w_0,q_0,x_0}$ plus an integer.  Since $\alpha_{w_0,q_0,x_0}$ is irrational, this shows that the shifts $T^{hq_0w^*_0} E_{w_0,q_0,a,b,c}$ are all disjoint.  As these events occur with the same probability, they must therefore be null events, giving the claim.
\end{proof}

We now analyze the improved structured solutions provided by Proposition \ref{furt-struct-2} further.  We begin with a simple observation, that limits the occurrence of non-trivial combinatorial cases:

\begin{lemma}[At most one non-trivial combinatorial case per coset]\label{comb-case} Let $$(\one_{A'}, \tilde g, W, (\varphi'_w)_{w \in W})$$ be as in Proposition \ref{furt-struct-2}, and let $q$ be sufficiently divisible.  Then for each coset $x_0+q\Z^2$, there is at most one $w_0 \in W$ for which $\varphi'_{w_0}$ is $\{0,1\}$-valued, but not periodic, on $x_0+q\Z^2$.
\end{lemma}

\begin{proof} If there were distinct $w_0, w_1 \in W$ for which $\varphi'_{w_0}, \varphi'_{w_1}$ were $\{0,1\}$-valued, but not periodic (hence not identically constant), on $x_0+q\Z^2$, then by Proposition \ref{furt-struct-2} we have $\varphi'_{w_0} = \varphi'_{w_1}$ on $x_0+q\Z^2$.  But as $\varphi'_{w_0}$ is $q_0w_0$-invariant and $\varphi'_{w_1}$ is $q_0w_1$-invariant, this implies that $\varphi'_{w_0}$ is periodic on $x_0+q\Z^2$, a contradiction.
\end{proof}

\subsection{Analyzing the equidistributed case via the moment method}

Now we turn to the analysis of the equidistributed case.  We make some crucial first and second moment calculations:

\begin{lemma}[First and second moments]\label{first-second} Let $(\one_{A'}, \tilde g, W, (\varphi'_w)_{w \in W})$ be as in Proposition \ref{furt-struct-2}, and let $q$ be sufficiently divisible.  If $(w_0,x_0) \in \Gamma_{q,\e}$, then for any $x'_0 \in x_0 + q\Z^2$ and natural number $m$ one has
$$ \lim_{N \to \infty} \frac{1}{N^2} \sum_{h \in \{1,\dots,N\}^2} \varphi'_{w_0}(x'_0 + mqh ) = \frac{1}{2}$$
and
$$ \lim_{N \to \infty} \frac{1}{N^2} \sum_{h \in \{1,\dots,N\}^2} \varphi'_{w_0}(x'_0 + mqh)^2 = \frac{1}{3}.$$
If furthermore $(w_1,x_0) \in \Gamma_{q,\e}$ with $w_1 \neq w_0$, then
$$ \lim_{N \to \infty} \frac{1}{N^2} \sum_{h \in \{1,\dots,N\}^2} \varphi'_{w_0}(x'_0 + mqh) \varphi'_{w_1}(x'_0 + mqh) = \frac{1}{4}.$$
\end{lemma}

Informally: the $\varphi'_{w_0}$ behave asymptotically like uncorrelated random variables of mean $1/2$ and variance $1/3 - (1/2)^2 = 1/12$.

\begin{proof}  Write $x'_0 = x_0 + qk$ for some $k \in \Z^2$.  From \eqref{equid'} and \eqref{ynm} we have
$$ \varphi'_{w_0}( x'_0 + mqh ) = P'_{w_0,q,x_0}(w_0 \wedge (mh+k)) \mod 1$$
for $h \in \Z^2$.  By the Weyl equidistribution theorem, the fractional parts of $P'_{w_0,q,x_0}(w_0 \wedge (mh+k))$ for $h\in \Z^2$ are equidistributed on the unit interval $[0,1]$.  Hence $\varphi'_{w_0}(x'_0 + mqh) \in [0,1]$ is equal to $\{P'_{w_0,q,x_0}(w_0 \wedge (mh+k))\}$ outside of a set of density zero, and the first two limits converge to $\int_0^1 t\ dt = \frac{1}{2}$ and $\int_0^1 t^2\ dt = \frac{1}{3}$ (resp.) as claimed.  For the third claim, we note that
$$ \varphi'_{w_1}( x'_0 + mqh ) = P'_{w_1,q,x_0}(w_1 \wedge (mh+k)) \mod 1.$$
Since $w_0,w_1$ are linearly independent, we conclude from a further appeal to the Weyl equidistirbution theorem that the fractional parts of $$P'_{w_0,q,x_0}(w_0 \wedge (x'_0 + mqh)),\; P'_{w_0,q,x_1}(w_1 \wedge (x'_0 + mqh)) \hbox{ for } h \in \Z^2$$
are \emph{jointly} equidistributed on the unit square $[0,1]^2$.  Hence the third limit converges to $\int_0^1 \int_0^1 t_0t_1\ dt_0 dt_1 = \frac{1}{4}$ as claimed.
\end{proof}

This gives several important consequences:

\begin{corollary}\label{moment-consequences} Let $(\one_{A'}, \tilde g, W, (\varphi'_w)_{w \in W})$ be as in Proposition \ref{furt-struct-2}, let $q$ be sufficiently divisible, and let $x_0 \in [q]^2$.  Let $W_{x_0,q,\e}$ be the set of $w_1 \in W$ such that $(w_1,x_0) \in \Gamma_{q,\e}$.
\begin{itemize}
    \item[(i)] If $(w_0,x_0) \in \Gamma_{q,\cc}$ is such that $\varphi'_{w_0}$ is non-periodic on $x_0+q\Z^2$, then $W_{x_0,q,\e}$ is empty.
    \item[(ii)] The size of $W_{x_0,q,\e}$ is either zero or three. In the latter case, the three functions $\varphi'_{w_1}$ sum up to $2-\one_A$ on $x_0+q\Z^2$.
    \item[(iii)] If $w_1 \in W_{x_0,q,\e}$, then the polynomial $P'_{w_1,q,x_0}$ is linear.
\end{itemize}
\end{corollary}

\begin{proof}  
We begin with (i).  By Proposition \ref{furt-struct-2}(a'), we have $\one_{A'} = 1 - \varphi'_{w_0}$ on $x_0+q\Z^2$.  By \eqref{a-rep-ind}, we conclude that
\begin{equation}\label{comb}
    1 - \tilde g = \sum_{w \in W \backslash \{w_0\}} \varphi'_w
\end{equation} 
on $x_0+q\Z^2$.  We already know that $\tilde g$ is $q_0\Z^2$-periodic and hence is constant on $x_0+q\Z^2$, while from Lemma \ref{comb-case} we know that for all the remaining combinatorial cases $(w,x_0) \in \Gamma_{q,\cc}$, $w \in W\backslash \{w_0\}$, the corresponding  functions $\varphi'_w$ are also periodic on $x_0+q\Z^2$.  Thus we can find a subcoset $x'_0 + mq\Z^2$ of $x_0+q\Z^2$ on which all these functions are constant. Hence, by \eqref{comb}, on such a coset the sum $\sum_{w_1 \in W_{x_0,q,\e}} \varphi'_w$ is equal to a constant $c_{x'_0+mq\Z^2}$.  Computing the first and second moments of this quantity on $x'_0+mq\Z^2$ using Lemma \ref{first-second}, we see that
$$ c_{x'_0+mq\Z^2} = \frac{|W_{x_0,q,\e}|}{2}$$
and
$$ c_{x'_0+mq\Z^2}^2 = \frac{|W_{x_0,q,\e}|}{3} + \frac{|W_{x_0,q,\e}|(|W_{x_0,q,\e}|-1)}{4} = \frac{|W_{x_0,q,\e}|^2}{4} + \frac{|W_{x_0,q,\e}|}{12}.$$
These equations are only consistent when $|W_{x_0,q,\e}|=c_{x'_0+mq\Z^2}=0$, giving (i).

For the remaining claims, we may assume by (i) that if $W_{x_0,q,\e}$ is non-empty then $\varphi'_{w_0}$ is periodic whenever $(w_0,x_0) \in \Gamma_{q,\cc}$.  From \eqref{a-rep-ind} we then conclude that, after passing to an subcoset $x'_0 + mq\Z^2$ for sufficiently divisible $m$, we have
\begin{equation}\label{onea}
\one_{A'} = c_{x'_0+mq\Z^2} - \sum_{w_1 \in W_{x_0,q,\e}} \varphi'_{w_1}
\end{equation}
on $x'_0 + mq\Z^2$, for some constant $c_{x'_0+mq\Z^2}$.  Computing first and second moments on $x'_0 + mq\Z^2$, by Lemma \ref{first-second} and \eqref{onea}, we conclude that
$$ \lim_{N \to \infty} \frac{1}{N^2} \sum_{h \in \{1,\dots,N\}^2} \one_{A'}(x'_0+mqh) = c_{x'_0+mq\Z^2} - \frac{|W_{x_0,q,\e}|}{2}$$
and
\begin{align*}
\lim_{N \to \infty} \frac{1}{N^2} \sum_{h \in \{1,\dots,N\}^2} \one_{A'}(x'_0+mqh)^2 &= c_{x'_0+mq\Z^2}^2 - c_{x'_0+mq\Z^2}|W_{x_0,q,\e}| \\
&\quad + \frac{|W_{x_0,q,\e}|}{3} + \frac{|W_{x_0,q,\e}|(|W_{x_0,q,\e}|-1)}{4}\\
&= \left(c_{x'_0+mq\Z^2} - \frac{|W_{x_0,q,\e}|}{2}\right)^2 + \frac{|W_{x_0,q,\e}|}{12}.
\end{align*}
Both left-hand sides are equal to the same quantity $\theta \in [0,1]$; thus $$\theta=\theta^2+\frac{|W_{x_0,q,\e}|}{12}.$$  So, since $\theta - \theta^2 \leq \frac{1}{4}$, we conclude that
$$
\frac{|W_{x_0,q,\e}|}{12} \leq \frac{1}{4},$$
giving $|W_{x_0,q,\e}| \leq 3$.  Furthermore, in order for equality to hold, we must have $\theta = \frac{1}{2}$, and hence $c_{x'_0+mq\Z^2} = \theta + \frac{|W_{x_0,q,\e}|}{2} = 2$, so that the three $\varphi'_{w_1}$ sum up to $2 - \one_{A'}$ on $x'_0+mq\Z^2$ thanks to \eqref{onea}, giving the second part of (ii) after letting $x'_0$ vary.

Now suppose that $w_1 \in W_{x_0,q,\e}$.  From \eqref{onea} we have, after passing to a subcoset $x'_0+mq\Z^2$ as before, that
$$ \varphi'_{w_1} = c_{x'_0+mq\Z^2} - \sum_{w'_1 \in W_{x_0,q,\e} \backslash \{w_1\}} \varphi'_{w'_1} \mod 1$$
on $x'_0+mq\Z^2$ and some constant $c_{x'_0+mq\Z^2}$.  Since $W_{x_0,q,\e} \backslash \{w_1\}$ has cardinality at most two, we can eliminate the right-hand side with at most two derivatives, and conclude that
$$ \partial_{h_1} \partial_{h_2} \varphi'_{w_1} = 0 \mod 1$$
on $x'_0+mq\Z^2$ for some $h_1,h_2 \in q\Z^2$ that are linearly independent of $w_1$.  This is incompatible with $P'_{w_1,q,x_0}$ having an irrational leading coefficient if $P'_{w_1,q,x_0}$ had degree at least two; since $P'_{w_1,q,x_0}$ is non-constant, it must therefore be linear, giving (iii).  Observe that if we had $|W_{x_0,q,\e}| < 3$, we could use just one derivative instead of two here, forcing $P'_{w_1,q,x_0}$ to be constant on $x'_0+mq\Z^2$, which, again, contradicts the definition of $\Gamma_{q,\e}$.  This completes the proof of (ii).
\end{proof}

After some final (and easy) ``cleaning'' steps, we arrive at a solution in ``normal form'':

\begin{proposition}[Structured solution in normal form]\label{Struct-sol} Let $f \in \ell^\infty(\Z^2,\Z)_{\c} \backslash \{0\}$, $A\subset \Z^2$ and $g \in \ell^\infty(\Z^2,\Z)_{\p}$ be such that $f*\one_A=g$.   Then there exists an indicator function structured solution
$(\one_{A'}, \tilde g', W, (\varphi'_w)_{w \in W})$ to the equation $f*a=g$ such that for all sufficiently divisible $q$, one has a partition $[q]^2 = \Sigma_{q,\cc} \cup \Sigma_{q,\e} \cup \Sigma_{q,(0)}$ with the following properties:
\begin{itemize}
    \item (Combinatorial case) If $x_0 \in \Sigma_{q,\cc}$, then there is a $w_{x_0,q} \in W$ such that on $x_0+q\Z^2$, $\varphi'_{w_{x_0,q}}$ is $\{0,1\}$-valued, all the other $\varphi'_w$ vanish, and $\tilde g' = 1$.
    \item (Equidistributed case) If $x_0 \in \Sigma_{q,\e}$, then on $x_0+q\Z^2$, $\tilde g'=2$, and there exists a three-element subset $W_{x_0, q,\e}$ of $W$, such that $\varphi'_w$ vanishes for all $w \in W \backslash W_{x_0,q,\e}$, and for each $w \in W_{x_0,q,\e}$ there exist real numbers $\alpha_{w, q, x_0}, \beta_{w, q, x_0}$ with $\alpha_{w, q, x_0}, \beta_{w, q, x_0}$, $1$ linearly independent over $\Q$, such that
    \begin{equation}\label{relation} \varphi'_{w}(x_0 + nqw_i + mqw^*_i) = \{\alpha_{w, q, x_0} m + \beta_{w, q, x_0}\}.
\end{equation}
    for all $n,m \in \Z$.  Equivalently, using \eqref{ynm}, one has
    \begin{equation}\label{rel-2}
    \varphi'_w(y) =  \left\{\alpha_{w, q, x_0} \left(w \wedge \frac{y-x_0}{q}\right) + \beta_{w, q, x_0}\right\}
\end{equation}
    for all $y \in x_0+q\Z^2$.
    \item (Zero case) If $x_0 \in \Sigma_{q,(0)}$ then on $x_0+q\Z^2$, the $\varphi'_w$ all vanish, and $\tilde g'$ is identically equal to either $0$ or $1$.
\end{itemize}
In particular, $\tilde g'$ is integer-valued (in fact, it takes values in $\{0,1,2\}$).
\end{proposition}

\begin{proof}  We take the structured solution coming from Proposition \ref{furt-struct-2}.  For $q$ sufficiently divisible and $x_0 \in [q]^2$, we know from Lemma \ref{comb-case} and Corollary \ref{moment-consequences} that exactly one of the following two  statements holds:
\begin{itemize}
    \item[(i)]  We have $(w,x_0) \in \Gamma_{q,\cc}$ for all $w \in W$, and all but at most one of the $\varphi'_w$ are periodic on $x_0+q\Z^2$.
    \item[(ii)] We have $(w_1,x_0) \in \Gamma_{q,\e}$ for exactly three $w_1 \in W$, and $\varphi'_w$ on $x_0+q\Z^2$ is constant for all other $w \in W$.
\end{itemize}
Note that by making $q$ more divisible if needed to subdivide the cosets further, we can replace ``periodic'' by ``constant'' in (i).

If we are in case (i) with all of the $\varphi'_w$ constant, we subtract\footnote{By performing this modification, the structured solution no longer lies in the space $X_{q_0,W,\tilde g}$ used in the proof of Proposition \ref{furt-struct-2}; however, we will have no further use for this space.  Similarly for the modifications indicated in the other cases.} those constants from $\tilde g'$ on this coset to make $\varphi'_w$ zero on $x_0 + q\Z^2$ for all $w \in W$, while keeping $\tilde g'$ $q\Z^2$-periodic.  From \eqref{a-rep-ind} we now have $\one_{A'} = \tilde g'$ on $x_0+q\Z^2$, so $\tilde g'$ must equal either $0$ or $1$ on this coset; we are now in the zero case and can place $x_0$ in $\Sigma_{q,(0)}$.

If we are in case (i) with $\varphi'_{w_0}$ non-constant for some $w_0\in W$, then all other $\varphi'_w$ must be constant, and we may similarly make $\varphi'_w$ zero for all $w \neq w_0$ on this coset (modifying $\tilde g'$ as needed).  Then from \eqref{a-rep-ind}, we have $\one_{A'}=\tilde g'-\varphi'_{w_0}$, and so by the $q\Z^2$-periodicity of $\tilde g'$, the constant value that $\tilde g'$ takes on $x_0+q\Z^2$ must be equal to $1$.  Thus we are now in the combinatorial case, and can place $x_0$ in $\Sigma_{q,\cc}$.

In case (ii), we again subtract constants to make $\varphi'_w$ zero for all $w \in W$ with $(w,x_0) \in \Gamma_{q,\cc}$ (modifying $\tilde g'$ as needed), and then by \eqref{a-rep-ind} and Corollary \ref{moment-consequences} (and the $q\Z^2$-periodicity of $\tilde g'$) we have $\tilde g' = 2$ on $x_0+q\Z^2$.  The claim \eqref{relation} then follows from \eqref{equid''} and Corollary \ref{moment-consequences}(iii), so we are in the equidistributed case and can place $x_0$ in $\Sigma_{q,\e}$.
\end{proof}

\section{Higher level PTC in \texorpdfstring{$\Z^2$}{Z^2}, II: describing the normalized solution}\label{highlevel-sec}

In this section, we conclude the proof of Theorem \ref{main-3}, following the strategy described in Section \ref{strategy-sec}. Recall that the first step was established in Section \ref{cleaning-sec}; thus, in this section we establish the remaining steps.

 \subsection{Describing the irrationality of the solution}
 Observe that the indicator function structured solution $(\one_{A'}, \tilde g', W, (\varphi'_w)_{w \in W})$ provided by Proposition \ref{Struct-sol} involves the irrational parameters $\alpha_{w,q,x}, \beta_{w,q,x}$ (for $x\in \Sigma_{q,\e}$ and $w\in W_{x,q,\e}$) associated to the equidistributed case.  This makes other key functions, such as the periodic functions $(\one_{x + \langle w \rangle} f) * \varphi'_w$ arising in Theorem \ref{structure-thm-2}(ii) take irrational values as well, and will almost certainly cause this structured solution to be  non-periodic. In preparation for making these quantities rational (in order to construct a periodic solution), we now capture all the irrationality inside a single map $\Phi \in \Hom_1(\Q^{d+1},\R)$, with the irrational parameters $\alpha_{w,q,x}, \beta_{w,q,x}$ being reinterpreted as images of rational proxies $\overline{\alpha}_{w,q,x}, \overline{\beta}_{w,q,x}$ via this map.

\begin{proposition}[Describing the irrationality]\label{irrat}  Let $(\one_{A'}, \tilde g', W, (\varphi'_w)_{w \in W})$ be the indicator function structured solution in Proposition \ref{Struct-sol}.  Then, for sufficiently divisible $q$, there is an injective map $\Phi \in \Hom_1(\Q^{d+1},\R)$ for some $d \geq 0$, with the following properties:
\begin{itemize}
    \item For $x \in \Sigma_{q,\e}$ and $w \in W_{x,q,\e}$, there exist $\overline{\alpha}_{w,q,x} \in \Z^{d+1}$ and $\overline{\beta}_{w,q,x} \in \Q^{d+1}$ such that $\alpha_{w,q,x} = \Phi(\overline{\alpha}_{w,q,x})$ and $\beta_{w,q,x} = \Phi(\overline{\beta}_{w,q,x})$.  Furthermore, $\overline{\alpha}_{w,q,x}$, $\overline{\beta}_{w,q,x}$, and $e_0$ are\footnote{Recall that $\{e_0,\dots,e_d\}$ is the standard basis for $\Q^{d+1}$.} linearly independent over $\Q$.
    \item For each $w \in W$ and $x+\langle w \rangle \in \Z^2/\langle w \rangle$, there is a $q\Z^2$-periodic function $\overline{g}_{w, x+\langle w \rangle} \colon \Z^2 \to \Q^{d+1}$ such that
    \begin{equation}\label{tog}
    \Phi \circ \overline{g}_{w, x+\langle w \rangle} = (\one_{x + \langle w \rangle} f) * \varphi'_w
    \end{equation}
    and also
    \begin{equation}\label{ge0}
    g e_0 = (f * \tilde g')e_0 - \sum_{w \in W} \sum_{x + \langle w \rangle \in \Z^2/\langle w \rangle} \overline{g}_{w, x+\langle w \rangle}.
    \end{equation}
\end{itemize}
\end{proposition}

\begin{proof}  Fix $q$ that meets the requirements of Proposition \ref{Struct-sol}. For each $w \in W$ and $x + \langle w \rangle \in \Z^2/\langle w\rangle$, set
$$\tilde g'_{w, x+\langle w \rangle} \coloneqq (\one_{x + \langle w \rangle} f) * \varphi'_w.$$
By Theorem \ref{structure-thm-2}(ii), $\tilde g'_{w,x+\langle w \rangle}$ is $q\Z^2$-periodic.  Convolving \eqref{a-rep-ind} with $f$ and using \eqref{fag-eq-ind}, we conclude that
\begin{equation}\label{gifg}
g = f * \tilde g' - \sum_{w \in W} \sum_{x + \langle w \rangle \in \Z^2/\langle w \rangle} \tilde g'_{w, x+\langle w \rangle}.
\end{equation}
Let $\Gamma$ be the $\Q$-vector subspace of $\R$ generated by $1$ and all the $\alpha_{w,q,x_0}, \beta_{w,q,x_0}$ with $x_0 \in \Sigma_{q,\e}$ and $w \in W_{x_0,q,\e}$.  Then $\Gamma$ is finitely generated and thus has a $\Q$-linear basis $1, \gamma_1,\dots,\gamma_d$ for some $d \geq 0$. From Proposition \ref{Struct-sol}, we see that the $\varphi'_w$ all take values in $\Gamma$, hence the $\tilde g'_{w, x+\langle w \rangle}$ do as well.

We then let $\Phi \colon \Q^{d+1} \to \R$ be the $\Q$-linear map with $\Phi(e_0) \coloneqq 1$ and $\Phi(e_i) \coloneqq \gamma_i$ for $i=1,\dots,d$; then $\Phi$ is injective with range $\Gamma$.  We can then uniquely write
$\alpha_{w,q,x_0} = \Phi(\overline{\alpha}_{w,q,x_0})$ and $\beta_{w,q,x_0} = \Phi(\overline{\beta}_{w,q,x_0})$ for some $\overline{\alpha}_{w,q,x_0}, \overline{\beta}_{w,q,x_0} \in \Q^{d+1}$ for all $x_0 \in \Sigma_{q,\e}$ and $w \in W_{x_0,q,\e}$, and for any $w \in W$ and $x \in \Z^2/\langle w \rangle$, write $\tilde g'_{w, x+\langle w \rangle} = \Phi \circ \overline{g}_{w, x+\langle w \rangle}$ for some $q\Z^2$-periodic $\overline{g}_{w, x+\langle w \rangle} \in \ell^\infty(\Z^2,\Q^{d+1})_\p$, which gives \eqref{tog}.  Since $\alpha_{w,q,x}, \beta_{w,q,x}, 1$ are linearly independent over $\Q$, we obtain, by construction of $\Phi$, that $\overline{\alpha}_{w,q,x}, \overline{\beta}_{w,q,x}, e_0$ are also linearly independent over $\Q$.

In the identity \eqref{gifg}, the functions $g$, $f$, and $\tilde g'$ are integer-valued, hence $f*\tilde g'$ is also, thus we can apply $\Phi$ to \eqref{gifg}:
$$
\Phi(ge_0) = \Phi\left( (f * \tilde g')e_0 - \sum_{w \in W} \sum_{x + \langle w \rangle \in \Z^2/\langle w \rangle} \overline{g}_{w, x+\langle w \rangle} \right).$$
And so, since $\Phi$ is injective, we conclude \eqref{ge0}.

To complete the proof, it remains to show that we can choose the vectors $\overline{\alpha}_{w,q,x_0}$ to be in $\Z^{d+1}$ rather than just in $\Q^{d+1}$ (which we currently have). Observe from an inspection of \eqref{relation} that we can replace $q$ with any multiple $kq$ of $q$,  by setting
$$
\alpha_{w,kq,x_0+aq} \coloneqq k \alpha_{w,q,x_0}$$
and
$$ \beta_{w,kq,x_0+aq} \coloneqq \beta_{w,q,x_0} + a \alpha_{w,q,x_0}$$
for any $x_0 \in [q]$ and $0 \leq a < k$.  This replacement does not alter the group $\Gamma$, and so we can also keep $\Phi$ unchanged, so that  $\overline{\alpha}_{w,kq,x_0} = k \overline{\alpha}_{w,q,x_0}$ for every $x_0\in [kq]^2$. Thus, by taking $q$ sufficiently divisible, we can clear all denominators, making all $\overline{\alpha}_{w,q,x_0}$ $x_0\in \Sigma_{q,\e}$, $w\in W_{x_0,q,\e}$ integer-valued.
\end{proof}

Henceforth we fix the indicator function structured solution $$(\one_{A'}, \tilde g', W, (\varphi'_w)_{w \in W})$$ provided by Proposition \ref{Struct-sol}, and fix a sufficiently divisible $q$; we will make no further modifications to $q$, and to simplify the notation we now omit $q$ from the subscripting.  We then have a partition
$$[q]^2 = \Sigma_{\cc} \cup \Sigma_{\e} \cup \Sigma_{(0)},$$
primitive vectors 
$$w_{x} \in W \text{ for } x \in \Sigma_{\cc}$$
and sets 
$$W_{x,\e}, \; x\in \Sigma_{\e},$$ and for each $w\in W_{x,\e}$, real numbers $\alpha_{w, x}, \beta_{w, x}$ and vectors $\overline{\alpha}_{w,x},\overline{\beta}_{w,x}\in \Q^{d+1}$ for a suitable $d\geq 0$, and $q\Z^2$-periodic functions $$\overline{g}_{w, x+\langle w \rangle}\colon \Z^2\to \Q^{d+1}$$ obeying all the properties of Propositions \ref{Struct-sol} and \ref{irrat} (for the indicated choice of $q$) for a suitable injective $\Q$-linear map $\Phi \in \Hom_1(\Q^{d+1}, \R)$.

\subsection{Rationalization ansatz}
The map $\Phi \in \Hom_1(\Q^{d+1},\R)$ from Proposition \ref{irrat} takes irrational values, which is one of the main reasons why the solutions to $f*a=g$ we currently possess are not periodic.  The strategy is then to use Proposition \ref{nondeg-real} to replace $\Phi$ by another map $\Phi^\Q \in \Hom_1(\Q^{d+1},\Q)$ only taking rational values, and modify the functions $\varphi'_w$ accordingly.  For this, we should first identify a locally definable predicate $P$, capturing the properties of our $\Phi$ as well as $\Phi^\Q$ that we wish to find. More precisely, suppose that we could locate a map $\Phi^\Q \in \Hom_1(\Q^{d+1},\Q)$, together with periodic functions $\varphi_w^\Q \in \ell^\infty(\Z^2,\Q)_{\p}$ that are also $qw$-periodic for each $w\in W$, obeying the indicator function condition
\begin{equation}\label{indicator-cond}
\tilde g'(x) - \sum_{w \in W} \varphi_w^\Q(x) \in \{0,1\}
\end{equation}
for all $x \in \Z^2$, as well as the sliced tiling condition
\begin{equation}\label{tog-rat}
   \Phi^\Q \circ \overline{g}_{w, x+\langle w \rangle} = (\one_{x + \langle w \rangle} f) * \varphi^\Q_w
\end{equation}
 for all $w \in W$ and $x + \langle w \rangle \in \Z^2/\langle w \rangle$. Applying $\Phi^\Q$ to \eqref{ge0}, we obtain
 $$
     g = f * \tilde g' - \sum_{w \in W} \sum_{x + \langle w \rangle \in \Z^2/\langle w \rangle} \Phi^\Q \circ \overline{g}_{w, x+\langle w \rangle}
    $$
and hence by \eqref{tog-rat} and summation in $x$
$$ g = f * \tilde g' - \sum_{w \in W} f * \varphi^\Q_w.$$
If we denote the $\{0,1\}$-valued expression in \eqref{indicator-cond} as $\one_{A'}(x)$, then by construction $\one_{A'}$ is periodic, and we now have
$$ g = f * \one_{A'}$$
giving the required periodic indicator function solution for Theorem \ref{main-3}.  Thus, to finish the proof of Theorem \ref{main-3}, it will suffice to locate a map $$\Phi^\Q \in \Hom_1(\Q^{d+1}, \Q)$$ and $qw$-periodic functions $\varphi_w^\Q \in \ell^\infty(\Z^2,\Q)_{\p}$ obeying the conditions \eqref{indicator-cond}, \eqref{tog-rat}.  For comparison, from \eqref{a-rep}, \eqref{tog} we see that the map $\Phi \in \Hom_1(\Q^{d+1},\R)$ and the functions $\varphi'_w \in \ell^\infty(\Z^2,\R)$ already obey the analogues of \eqref{indicator-cond}, \eqref{tog-rat}, but are real-valued instead of rational-valued, and furthermore the $\varphi'_w$ are not known to be periodic (although they are still $qw$-periodic).

In order to avoid the discontinuities of the fractional part operator $\{\}$ at a later stage of the argument, we will also require $\Phi^\Q$ to be \emph{non-degenerate} in the sense that for every $x_0 \in \Sigma_{\e}$ and $w \in W_{x_0,q,\e}$, the quantity $\Phi^\Q(\overline{\beta}_{w, q, x_0})$ is not an integer linear combination of $\Phi^\Q(\overline{\alpha}_{w, q, x_0})$ and $1$.  Fortunately, this condition will eventually be provided for us by Proposition \ref{nondeg-real}, together with the hypothesis that $\Phi$ is injective.

The functions $\varphi^\Q_w$ we will use will be of a particular form, defined in terms of the (non-degenerate) map $\Phi^\Q$ and some additional combinatorial data $A_{x_0}$ for $x_0 \in \Sigma_{\cc}$.  More precisely, we adopt the following ansatz for the functions $\varphi^\Q_w$ on the various cosets $x_0 + q\Z^2$, $x_0 \in [q]^2$ of $q\Z^2$:
\begin{itemize}
    \item (Combinatorial case) If $x_0 \in \Sigma_{\cc}$, then $\varphi^\Q_{w_{x_0}} = \one_{A_{x_0}}$ on $x_0 + q\Z^2$ for some $w_{x_0}\in W$ and some periodic set $A_{x_0} \subset x_0+q\Z^2$ which is also $qw_{x_0}$-periodic, and $\varphi^\Q_w$ vanishes identically  on $x_0+q\Z^2$ for all  $w \in W\backslash \{w_{x_0}\}$.
    \item (Equidistributed case) If $x_0 \in \Sigma_{\e}$, then 
    $\varphi^\Q_w$ vanishes identically on $x_0+q\Z^2$ for all $w \in W \backslash W_{x_0,\e}$, and for each $w \in W_{x_0,\e}$, $y\in x_0+q\Z^2$ we have
    $$ \varphi^\Q_{w}(y) = \left\{\Phi^\Q(\overline{\alpha}_{w, x_0}) \left(w\wedge \frac{y-x_0}{q} \right) + \Phi^\Q(\overline{\beta}_{w, x_0}) \right\}.
$$
Note that if $\Phi^\Q$ is non-degenerate, then the expression inside the fractional part is never an integer.
    \item (Zero case) If $x_0 \in \Sigma_{(0)}$, then $\varphi^\Q_w$ vanishes identically on $x_0+q\Z^2$ for all $w \in W$.
\end{itemize}
Note that the functions $\varphi^\Q_w$ will necessarily lie in $\ell^\infty(\Z^2,\Q)_{\p}$, since they vanish in the zero case, are assumed to be periodic and $\{0,1\}$-valued in the combinatorial case, and are given by the fractional part of a linear form with rational coefficients in the equidistributed case.  For similar reasons, the $\varphi^\Q_w$ will necessarily be $qw$-periodic.  For comparison, we see from Proposition \ref{irrat} that our given functions $\varphi'_w$ also obey a similar ansatz, but with $\Phi^\Q$ replaced by $\Phi$, and with $A_{x_0}$ replaced by a set $A'_{x_0} \subset x_0+q\Z^2$ which remains $qw_{x_0}$-periodic, but is no longer assumed to be periodic ($A'_{x_0}= (x_0+q\Z^2) \backslash \supp(\varphi'_{w_{x_0}})$).

Let us now record what the conditions \eqref{indicator-cond}, \eqref{tog-rat} reduce to under this ansatz.  We begin with \eqref{indicator-cond}.  If $x \in \Z^2$, then $x \in x_0 + q\Z^2$ for some $x_0 \in [q]^2$.  The ansatz ensures that the constraint \eqref{indicator-cond} is automatically satisfied in the combinatorial case  $x_0 \in \Sigma_{\cc}$ and the zero case $x_0 \in \Sigma_{(0)}$ thanks to the properties of $\tilde g'$ in these cases from Proposition \ref{Struct-sol}.  In the equidistributed case $x_0 \in \Sigma_{\e}$, since $\tilde g'=2$ on $x_0+q\Z^2$ in this case, the indicator constraint simplifies to
\begin{equation}\label{equi-con}
\sum_{w \in W_{x_0,\e}} \{\Phi^\Q(\overline{\alpha}_{w, x_0}) (w \wedge y) + \Phi^\Q(\overline{\beta}_{w, x_0})\} \in \{1,2\}
\end{equation}
for all $y \in \Z^2$.
For sake of comparison to the indicator function structured solution
$(\one_{A'}, \tilde g', W, (\varphi'_w)_{w \in W})$, by \eqref{rel-2} we have the condition
\begin{equation}\label{equi-con-compare}
\begin{split}
\sum_{w \in W_{x_0,\e}} \{\Phi(\overline{\alpha}_{w, x_0}) (w \wedge y) + \Phi(\overline{\beta}_{w, x_0})\} &=
\sum_{w \in W} \varphi'_w(x_0+qy) \\
&= \tilde g'(x_0+qy) - \one_{A'}(x_0+qy) \\
&= 2 - \one_{A'}(x_0+qy) \\
&\in \{1,2\}
\end{split}
\end{equation}
for all $y \in \Z^2$.

\subsection{Eliminating the combinatorial data}
We remark that the combinatorial data $A_{x_0}$ plays no role in the reduced indicator constraint \eqref{equi-con}, and thus only affects the sliced tiling constraint \eqref{tog-rat}.  We now perform some transformations of this constraint to eliminate the combinatorial data $A_{x_0}$ entirely.  First we observe that to verify \eqref{tog-rat}, it suffices to do so in the case where
the coset $x+\langle w \rangle$ intersects the support $F$ of $f$; by shifting $x$, we may assume that it is $x$ itself that lies in $F$.  Evaluating the constraint \eqref{tog-rat} for a given $w \in W$ and $x \in F$, and at a given location $y+x$ of $\Z^2$, we can expand the constraint as
\begin{equation}\label{constraint2Q}
    \Phi^\Q \circ \overline{g}_{w, x+\langle w \rangle}(y+x) = \sum_{h \in \Z} f(x+hw) \varphi^\Q_w(y - hw).
\end{equation} 
Using the ansatz for $\varphi^\Q_w$, we rewrite this constraint to place all terms involving the combinatorial data $A_{x_0}$ on the right-hand side, and all the terms involving the map $\Phi^\Q$ on the left-hand side:
\begin{align*}
    &\Phi^\Q \circ \overline{g}_{w, x+\langle w \rangle}(y+x) \\
    & \quad - \sum_{\stackrel{x_0 \in \Sigma_{\e}}{w \in W_{x_0,\e}}}\;
\sum_{\stackrel{h \in \Z}{y-hw \in x_0+q\Z^2}} f(x+hw)  \left\{\Phi^\Q(\overline{\alpha}_{w, q,x_0}) \left(w\wedge \frac{y-hw-x_0}{q}\right) + \Phi^\Q(\overline{\beta}_{w, q,x_0})\right\}  \\
&\quad = \sum_{\stackrel{x_0 \in \Sigma_{\cc}}{w_{x_0} = w}}\; \sum_{\stackrel{h \in \Z}{y-hw \in x_0+q\Z^2}} f(x+hw) \one_{A_{x_0}}(y - hw).
\end{align*}
We can write $y = (nq+a)w + (mq+b)w^*$ for some $a,b \in [q]$ and $n,m \in \Z$, where the constraint $y-hw \in x_0+q\Z^2$ now becomes $(a-h)w + bw^* \in x_0 + q\Z^2$, which implies on taking wedge products with $w$ that the quantity
\begin{equation}\label{vwinZ}
k_{w,b,x_0}\coloneqq \frac{b - (w \wedge x_0)}{q}\in \Z
\end{equation}
is an integer.  
Using the periodicity properties of $\overline{g}_{w,x+\langle w \rangle}$ and $A_{x_0}$, the above constraint then can be written as
\begin{align*}
S_{w,b,a,x}(\Phi^\Q, m) = \sum_{\stackrel{x_0 \in \Sigma_{\cc}}{w_{x_0} = w}}\; \sum_{\stackrel{h \in \Z}{(a-h)w+bw^* \in x_0+q\Z^2}} f(x+hw) \one_{A_{x_0}}((a-h)w + (mq+b) w^*),
\end{align*}
where  for any $\Q$-linear map $\tilde \Phi \colon \Q^{d+1} \to \R$, 
\begin{equation}\label{swb-def}
\begin{split}
S_{w,b,a,x}(\tilde \Phi, m) &\coloneqq \tilde \Phi \circ \overline{g}_{w, x+\langle w \rangle}(aw+bw^*+x) \\
&\quad - \sum_{\stackrel{x_0 \in \Sigma_{\e}}{w \in W_{x_0,\e}}}\;
\sum_{\stackrel{h \in \Z}{(a-h)w+bw^* \in x_0+q\Z^2}} f(x+hw) \times \\
&\quad \left\{\tilde \Phi(\overline{\alpha}_{w, x_0}) \left(m+k_{w,b,x_0}\right) + \tilde \Phi(\overline{\beta}_{w, x_0})\right\}.
\end{split}
\end{equation}
This constraint is required to hold for all $w \in W$, $a,b \in [q]$, $x \in F$, and $m \in \Z$ (the parameter $n$ now plays no role).  We collect the different values of $a, x$ in a tuple, and rewrite the desired constraint as
\begin{equation}\label{sliced-eq}
S_{w,b}( \Phi^\Q, m ) = T_{w,b}( (\{ n \in \Z: nw_{x_0} + (mq+b) w_{x_0}^* \in A_{x_0} \})_{x_0 \in \Sigma_{\cc}} ) 
\end{equation}
for all $w \in W$, $b \in [q]$, and $m \in \Z$, where
$$ S_{w,b}(\tilde \Phi,m ) \coloneqq (S_{w,b,a,x}(\tilde \Phi,m))_{(a,x) \in [q] \times F} \in \R^{[q] \times F}$$
and for any tuple $(B_{x_0})_{x_0 \in \Sigma_{\cc}}$ of $q$-periodic subsets $B_{x_0}$ of $\Z$, $$T_{w,b}((B_{x_0})_{x_0 \in \Sigma_{\cc}}) \in \Z^{[q] \times F}$$ denotes the tuple
$$
T_{w,b}((B_{x_0})_{x_0 \in \Sigma_{\cc}}) \coloneqq
\left( \sum_{\stackrel{x_0 \in \Sigma_{\cc}}{w_{x_0} = w}} \; \sum_{\stackrel{h \in \Z}{(a-h)w+bw^*-x \in x_0+q\Z^2}} f_{x,w}(h) \one_{B_{x_0}}(a - h)\right)_{(a,x) \in [q] \times F},$$ 
where $f_{x,w}(h)\coloneqq f(x+hw)$.
The reason for writing the constraint \eqref{constraint2Q} in the form \eqref{sliced-eq} is that it decouples the various components of the combinatorial sets $A_{x_0}$, to the point where we can eliminate the role of these sets entirely, thanks to Lemma \ref{period-extend}.  To see this, let $\Omega_{w,b} \subset \Z^{[q] \times F}$ denote the set of all tuples of the form $T_{w,b}((B_{x_0})_{x_0 \in \Sigma_{\cc}})$ for some $q$-periodic subsets $B_{x_0}$ of $\Z$.   Note that $\Omega_{w,b}$ is a finite set, since there are only $2^{q|\Sigma_{\cc}|}$ possible choices for  the $B_{x_0}$.  Then clearly, the assumption \eqref{sliced-eq} on $\Phi^{\Q}$ implies that
\begin{equation}\label{sliced-eq'} 
S_{w,b}( \Phi^\Q, m ) \in \Omega_{w,b}
\end{equation}
for all $w \in W$, $b \in [q]$, and $m \in \Z$. Conversely, suppose that \eqref{sliced-eq'} holds for all $w,b,m$.  Thus, by definition of $\Omega_{w,b}$, we can write
$$
S_{w,b}( \Phi^\Q, m ) = T_{w,b}((B_{x_0,w,b,m})_{x_0 \in \Sigma_{\cc}})$$
for some $q$-periodic sets $B_{x_0,w,b,m}$. Since $\Phi^\Q$ takes rational values, we see from \eqref{swb-def} that the tuple $S_{w,b}( \Phi^\Q, m )$ is periodic in $m$ for each $w,b$; by Lemma \ref{period-extend}, we may therefore select the $B_{x_0,w,b,m}$ to also be periodic in $m$ for each choice of $w,b,x_0$. If we  now define $A_{x_0}$ for any $x_0 \in \Sigma_{\cc}$ by the formula
$$ A_{x_0} \coloneqq \{ nw_{x_0} + (mq+b) w_{x_0}^*: b \in [q], m \in \Z, n \in B_{x_0,w_{x_0},b,m} \},$$
it is then a routine matter to check that $A_{x_0}$ is both periodic and $qw_{x_0}$-periodic, and that 
$$ T_{w,b}((B_{x_0,w,b,m})_{x_0 \in \Sigma_{\cc}}) = T_{w,b}( (\{ n: nw_{x_0} + (mq+b) w_{x_0}^* \in A_{x_0} \})_{x_0 \in \Sigma_{\cc}} ) $$
for all $w \in W$, $b \in [q]$, and $m \in \Z$, so that \eqref{sliced-eq} holds.  Thus we have reduced the sliced tiling constraint \eqref{tog-rat} to a constraint \eqref{sliced-eq'} that does not involve the combinatorial $A_{x_0}$.  For comparison, if we repeat the above analysis with $\Phi^\Q$, $\varphi^\Q_w$, $A_{x_0}$ replaced by $\Phi$, $\varphi'_w$, $A'_{x_0}$ (and \eqref{tog-rat} replaced by \eqref{tog}), we conclude that
\begin{equation}\label{swb-orig}
S_{w,b}( \Phi, m ) \in \Omega_{w,b}
\end{equation}
for all $w \in W$, $b \in [q]$, and $m \in \Z$, with the witnessing sets $B_{x_0}$ now taken to be the sets $\{n: nw_{x_0} + (mq+b)w^*_{x_0} \in A'_{x_0}\}$, which remain $q$-periodic thanks to the $qw_{x_0}$-periodicity of $A'_{x_0}$ (although here the sets $A'_{x_0}$ are not necessarily fully periodic).  

To summarize the previous discussion: the task of proving Theorem \ref{main-3} has now been reduced to locating a non-degenerate map $\Phi^\Q \in \Hom_1(\Q^{d+1},\Q)$ obeying the following two constraints:
\begin{itemize}
    \item[(a)] For any $x_0 \in \Sigma_{\e}$, we have the identity \eqref{equi-con} for all $y \in \Z^2$.
    \item[(b)] For all $w \in W$, $b \in [q]$, we have the constraint \eqref{sliced-eq'} for all $m \in \Z$.
\end{itemize}
For comparison, we see from \eqref{equi-con-compare}, \eqref{swb-orig} that we already possess an injective map $\Phi \in \Hom_1(\Q^{d+1},\R)$ that obeys both (a) and (b).

\subsection{Reducing to a finite number of constraints}

We have successfully eliminated all the combinatorial data, so that the only remaining unknown in the ansatz above is the map $\Phi^\Q \in \Hom_1(\Q^{d+1},\Q)$.  We are close to being able to apply Proposition \ref{nondeg-real}, as the conditions \eqref{equi-con}, \eqref{sliced-eq'} are (locally) definable in the language ${\mathcal L}$ of linear inequalities with rational coefficients for any \emph{fixed} choice of $y \in \Z^2$ or $m \in \Z$.  However, because the parameters $y$ and $m$ range over infinite sets, we are not quite able to invoke this proposition yet, as it requires the number of constraints to be finite.  Fortunately, we can use equidistribution theory to reduce the number of constraints from infinite to finite, as long as we avoid the discontinuities of the fractional part operator $\{\}$.

To see this, we first look at \eqref{equi-con} for a fixed $x_0 \in \Sigma_{\e}$, which is a little simpler notationally to work with\footnote{In fact, this condition is simple enough that one could in principle completely classify all solutions to this equation, but we adopt here a more abstract ``model-theoretic'' approach, as it will also work for the more complicated constraint \eqref{sliced-eq'}.}.  If we expand $w, y \in \Z^2$ out in coordinates as $w = (w_1,w_2)$, $y = (y_1,y_2)$, and also write $\overline{\alpha}_{w,x_0} = \sum_{i=0}^d \overline{\alpha}_{w,x_0,i} e_i$ for some integers $\overline{\alpha}_{w,x_0,i}$,
then the condition \eqref{equi-con} can be rewritten as
$$
F_{x_0}( (y_1 \Phi^\Q(e_i) \mod 1)_{i=0}^d, (y_2 \Phi^\Q(e_i) \mod 1)_{i=0}^d, \Phi^\Q ) \in \{1,2\}$$
for all $y_1,y_2 \in \Z$, where the piecewise continuous function $$F_{x_0} \colon (\R/\Z)^{d+1} \times (\R/\Z)^{d+1}\times \Hom_1(\Q^{d+1},\R) \to \R$$ is defined as
$$ F_{x_0}( (r_i)_{i=0}^d, (s_i)_{i=0}^d, \tilde \Phi )
\coloneqq \sum_{w \in W_{x_0,\e}} \left\{ \sum_{i=0}^d 
\tilde \Phi(\overline{\alpha}_{w, x_0,i}) (w_1 s_i - w_2 r_i) + \tilde \Phi(\overline{\beta}_{w,x_0}) \right\};$$ 
note that the right-hand side is $1$-periodic in each of the $r_i, s_i$, so $F_{x_0}$ is well-defined on the indicated domain, although it can have jump discontinuities whenever one has
\begin{equation}\label{jump}
\sum_{i=0}^d 
\tilde \Phi(\overline{\alpha}_{w, x_0,i}) (w_1 s_i - w_2 r_i) + \tilde \Phi(\overline{\beta}_{w,x_0}) = 0 \mod 1
\end{equation}
for some $w \in W_{x_0,\e}$.  Similarly, the identity \eqref{equi-con-compare} can be written as
$$
F_{x_0}( (y_1 \Phi(e_i) \mod 1)_{i=0}^d, (y_2 \Phi(e_i) \mod 1)_{i=0}^d, \Phi) \in \{1,2\}$$
for all $y_1,y_2 \in \Z$.  But as $\Phi \in \Hom_1(\Q^{d+1},\R)$ is injective, the real numbers $\Phi(e_0),\Phi(e_1),\dots,\Phi(e_d)$ are linearly independent over $\Q$.  By the equidistribution theorem, this implies that the orbit
$$ \{ ((y_1 \Phi(e_i) \mod 1)_{i=0}^d, (y_2 \Phi(e_i) \mod 1)_{i=0}^d): y_1,y_2 \in \Z \}$$
is dense in $(\R/\Z)^{d+1}\times (\R/\Z)^{d+1}$.  In particular, by taking limits, we conclude that the constraint
$$
F_{x_0}( (r_i)_{i=0}^d, (s_i)_{i=0}^d, \Phi ) \in \{1,2\}$$
holds for all $r_i, s_i \in \R/\Z$ that avoids the discontinuities \eqref{jump} of $F_{x_0}$.  To put it another way, if we let $\Lambda_{x_0} \subset \Hom_1(\Q^{d+1},\R)$ denote the set of all $\tilde \Phi \in \Hom_1(\Q^{d+1},\R)$ with the property that
\begin{equation}\label{jump-2}
F_{x_0}( (r_i)_{i=0}^d, (s_i)_{i=0}^d, \tilde \Phi) \in \{1,2\}
\end{equation}
holds whenever $r_i, s_i \in \R/\Z$ avoid the jump discontinuities \eqref{jump}, then
$$ \Phi \in \Lambda_{x_0}.$$
Conversely, if the rational map $\Phi^\Q$ is non-degenerate and is such that
\begin{equation}\label{phiq-lambda}
 \Phi^\Q \in \Lambda_{x_0},
\end{equation} 
then by reversing the above analysis we will obtain \eqref{equi-con}; the non-degeneracy hypothesis allows us to avoid the discontinuities \eqref{jump}.   

We can identify the torus $(\R/\Z)^{d+1}$ with its fundamental domain $[0,1)^{d+1}$, so that $F_{x_0}$ is now a function on $[0,1)^{d+1} \times [0,1)^{d+1} \times \Hom_1(\Q^{d+1},\R)$.  For any compact subset $K$ of $\Hom_1(\Q^{d+1},\R)$, we consider the predicate \eqref{jump-2} 
with $r_i,s_i \in [0,1)$ and $\tilde \Phi \in K$.  Under these restrictions, the left-hand side of \eqref{jump} is bounded.  Because of this, the predicate \eqref{jump-2} is expressible under these restrictions as a (moderately complicated) first-order sentence (depending on $K$) in the language ${\mathcal L}$ of linear inequalities with rational coefficients over the reals.  Using quantifier elimination in this language (see Remark \ref{lin-program}), the set $\Lambda_{x_0}$, when restricted to $K$, can be defined in $K$ by as a zeroth order sentence in ${\mathcal L}$; that is to say, $\Lambda_{x_0}$ is locally definable in $ {\mathcal L}$.

To summarize the above discussion, we have replaced the condition \eqref{equi-con}, which is the conjunction of an infinite number of linear constraints, with the condition \eqref{phiq-lambda}, which is locally definable in ${\mathcal L}$.

We now perform a similar substitution for the condition \eqref{sliced-eq'} for a given $w \in W$ and $b \in [q]$.  If the variable $m$ were restricted to a finite set, then this condition would be expressible as a finite sentence in ${\mathcal L}$; but $m$ is ranging over the infinite set $\Z$.  To ``compactify'' the role of this variable, we observe that
$$ \tilde \Phi(\overline{\alpha}_{w,x_0}) m = \sum_{i=0}^d \overline{\alpha}_{w, x_0,i} p_i(\tilde \Phi,m) \mod 1$$
for any $x_0 \in \Sigma_{\e}$, where $p_i(\tilde \Phi,m) \coloneqq \tilde \Phi(e_i) m \mod 1$.  Thus, by \eqref{swb-def}, we can write
$$ S_{w,b}(\tilde \Phi, m) = S'_{w,b}(  \tilde \Phi, (\tilde \Phi(e_i) m \mod 1)_{i=0}^d)$$
for any $\tilde \Phi \in \Hom_1(\Q^{d+1},\R)$, where $S'_{w,b} \colon \Hom_1(\Q^{d+1},\R)\times  (\R/\Z)^{d+1}  \to \R^{[q] \times F}$ is the piecewise linear function
$$ S'_{w,b}(\tilde \Phi , (p'_i)_{i=0}^d ) \coloneqq ( S'_{w,b,a,x}(\tilde \Phi, (p'_i)_{i=0}^d))_{(a,x) \in [q] \times F}
$$
and
\begin{align*}
S'_{w,b,a,x}(\tilde \Phi, (p'_i)_{i=0}^d) &\coloneqq \tilde \Phi \circ \overline{g}_{w, x+\langle w \rangle}(aw+bw^*+x) \\
&\quad - \sum_{\stackrel{x_0 \in \Sigma_{\e}}{w \in W_{x_0,\e}}}\;
\sum_{\stackrel{h \in \Z}{(a-h)w+bw^* \in x_0+q\Z^2}} f(x+hw) \times \\
&\quad \left\{\sum_{i=0}^d \overline{\alpha}_{w, x_0,i} p'_i + \tilde \Phi(\overline{\alpha}_{w, x_0}) k_{w,b,x_0} + \tilde \Phi(\overline{\beta}_{w,x_0})\right\}.
\end{align*}
Similarly to the functions $F_{x_0}$, the $S'_{w,b}$ are continuous as long as the input avoids the jump discontinuities, which occur when
\begin{equation}\label{jump-22}
\sum_{i=0}^d \overline{\alpha}_{w, x_0,i} p'_i + \tilde \Phi(\overline{\alpha}_{w,x_0}) k_{w,b,x_0} + \tilde \Phi(\overline{\beta}_{w, x_0}) = 0 \mod 1
\end{equation}
for some $x_0 \in \Sigma_{\e}$ with $w \in W_{x_0,q,\e}$.
From \eqref{swb-orig} we have
$$ S'_{w,b}( \Phi , (\Phi(e_i) m \mod 1)_{i=0}^d) \in \Omega_{w,b}$$
for all $m \in \Z$.  By the injectivity of $\Phi$ and the equidistribution theorem, we conclude that
$$ S'_{w,b}(\Phi , (p_i)_{i=0}^d) \in \Omega_{w,b}$$
for all $(p_i)_{i=0}^d \in (\R/\Z)^{d+1}$ avoiding the jump discontinuities \eqref{jump-22}.  That is to say, if we define $\Lambda'_{w,b} \subset \Hom_1(\Q^{d+1},\R)$ to be the set of all $\tilde \Phi \in \Hom_1(\Q^{d+1},\R)$ for which
$$ S'_{w,b}(\tilde \Phi , (p_i)_{i=0}^d ) \in \Omega_{w,b}$$
$(p_i)_{i=0}^d \in (\R/\Z)^{d+1}$ avoiding the jump discontinuities \eqref{jump-22}, then we have
$$ \Phi \in \Lambda'_{w,b}.$$
Conversely, if $\Phi^\Q \in \Hom_1(\Q^{d+1},\Q)$ is non-degenerate and obeys
\begin{equation}\label{phiq-lambdap}
\Phi^\Q \in \Lambda'_{w,b}
\end{equation}
then by reversing the above steps we will conclude that \eqref{sliced-eq'} holds, since the non-degeneracy hypothesis allows us to avoid the discontinuities in \eqref{jump-22}.  

As before, if we restrict $\tilde \Phi$ to lie in a compact subset $K$ of $\Hom_1(\Q^{d+1},\R)$ (and also identify $(\R/\Z)^{d+1}$ with its fundamental domain $[0,1)^{d+1}$), the property of $\tilde \Phi$ lying in $\Lambda'_{w,b}$ is a (quite complicated) first-order sentence in the language of linear inequalities over the rationals in the reals, so by quantifier elimination as before, $\Lambda'_{w,b}$ is locally definable in ${\mathcal L}$.

To summarize, we have reduced our problem to that of finding a non-degenerate map $\Phi^\Q \in \Hom_1(\Q^{d+1},\Q)$ that obeys \eqref{phiq-lambda} for all $x_0 \in W_{\e}$ and \eqref{phiq-lambdap} for all $w \in W$ and $b \in [q]$.  These conditions were already satisfied by the injective map $\Phi \in \Hom_1(\Q^{d+1},\R)$, and all of the conditions are locally definable in ${\mathcal L}$.  Theorem \ref{main-3} thus follows from Proposition \ref{nondeg-real}.

\end{document}